%% file: Main.tex
\documentclass[letterpaper, 10 pt, conference]{ieeeconf}

\IEEEoverridecommandlockouts
\usepackage{amsmath,mathptmx,amssymb}
\usepackage{tikz}
\usepackage{pgfplots}
\pgfplotsset{compat=newest}
\usetikzlibrary{backgrounds,calc,positioning}
\newtheorem{assumption}{Assumption}

\newtheorem{lemma}{Lemma}

\newtheorem{theorem}{Theorem}

\newtheorem{remark}{Remark}

\newcommand{\argmin}{\operatorname*{arg\,min}}
\usepackage{graphicx}
\usepackage{amsbsy}
\usepackage{caption}
\usepackage{subcaption}
\usepackage{xcolor}
\usepackage{mathtools}

\makeatletter
\let\NAT@parse\undefined
\makeatother

\usepackage{cite}
\usepackage[
colorlinks=true,
linkcolor=blue,
citecolor=blue,
urlcolor=blue
]{hyperref}
\usepackage{ifthen}
\newboolean{showcomments}
\setboolean{showcomments}{true}

\newcommand{\fethi}[1]{\ifthenelse{\boolean{showcomments}}
	{\textcolor{blue}{(Fethi says: #1)}}{}}
\newcommand{\bruce}[1]{\ifthenelse{\boolean{showcomments}}
	{\textcolor{red}{(BL: #1)}}{}}

\newcommand{\Tr}{\operatorname{Tr}}
\newcommand{\cov}[1]{\Sigma^{#1}}

\title{\LARGE \bf
Optimal input design via Frank--Wolfe
}
\author{Fethi Bencherki, Bruce Lee, Nikolai Matni, Anders Rantzer
\thanks{F. Bencherki and A. Rantzer are with the Department of Automatic Control, Lund University, Sweden. Email: {\tt\small \{fethi.bencherki, anders.rantzer\}@control.lth.se}. They are members of the ELLIIT Strategic Research Area at Lund University. This project received funding from the European Research Council (ERC) under Grant Agreements No.~834142 (ScalableControl) and No.~101199738 (DualControl), and was also partially supported by the Wallenberg AI, Autonomous Systems and Software Program (WASP), funded by the Knut and Alice Wallenberg Foundation.}
\thanks{B. Lee is with the ETH AI Center, ETH Z\"urich, Z\"urich, Switzerland. Email: {\tt\small bruce.lee@ai.ethz.ch}. He is funded by an ETH AI Center Postdoctoral Fellowship and NCCR Automation. This work was supported as a part of NCCR Automation, a National Centre of Competence in Research, funded by the Swiss National Science Foundation (grant number 51NF40\_225155).}
\thanks{N. Matni is with the Department of Electrical and Systems Engineering, University of Pennsylvania, Philadelphia, PA, USA. Email: {\tt\small nmatni@seas.upenn.edu}.}
}

\begin{document}

\maketitle

\begin{abstract}
We study optimal input design over a finite horizon for linear dynamical systems. The goal is to minimize a weighted inverse-covariance (information) criterion subject to an energy budget. The set of covariances achievable by causal policies is convex but lacks a tractable explicit description, ruling out projection-based methods. We show that Frank--Wolfe applies naturally: each linear minimization subproblem is a budget-constrained finite-horizon linear quadratic (LQ) problem, solvable by a Riccati recursion and one-dimensional bisection over a Lagrange multiplier. Using smoothness of the objective over the feasible set, we establish an $\mathcal{O}(1/M)$ convergence rate for the objective value, while strong convexity yields an $\mathcal{O}(1/\sqrt{M})$ rate for the iterates. We further extend the framework to input design for system identification with unknown dynamics and adaptive online LQR, and illustrate the approach numerically.
\end{abstract}
\section{Introduction}

Experiment design for dynamical systems asks how control inputs should be chosen so that the data they generate are maximally informative for a downstream task, such as identifying the system or controlling it well \cite{wahlberg2010optimal, wagenmaker2020active, wagenmaker2021task}. For linear systems with least-squares estimation, the quality of the resulting parameter estimates is governed by the inverse of the state--input covariance matrix accumulated over the experiment. This motivates optimizing a scalarization of the inverse covariance subject to an energy budget. In this work, we take a weighted trace as our choice of scalarization, corresponding to the classical A-optimal design criterion.

The resulting optimization problem is convex when viewed as a problem over achievable covariance matrices. The difficulty is that the feasible set, consisting of the covariances realizable by causal policies interacting with the dynamics, is only implicitly defined, so projected gradient methods are impractical. Following \cite{mutny2023active, wagenmaker2024optimal}, we instead use the Frank--Wolfe (conditional gradient) method \cite{frank1956algorithm, jaggi2013revisiting}, whose iterations require only linear minimization over the feasible set. The key structural fact is that linear functions of the achievable covariance are exactly expected quadratic costs, so each Frank--Wolfe subproblem is a finite-horizon LQ problem with an energy constraint.

\textbf{Contributions.} We provide a compact, self-contained treatment of this approach:
\begin{enumerate}
\item[\textit{(i)}] We formulate the budget-constrained covariance design problem, accounting for prior information, and establish smoothness and strong convexity over the achievable set with explicit constants.
\item[\textit{(ii)}] We show that the resulting linear minimization problem reduces to a Riccati recursion and scalar bisection.
\item[\textit{(iii)}] We prove an $\mathcal{O}(1/M)$ objective rate and an $\mathcal{O}(1/\sqrt{M})$ iterate rate.
\item[\textit{(iv)}] We outline extensions to system identification with unknown dynamics and adaptive online LQR, with numerical illustrations.
\end{enumerate}

\section{Problem formulation}
Consider the fully observed linear time-invariant system
\begin{align}\label{eq:system}
x_{t+1}=Ax_t+Bu_t+w_t,\quad x_0=0,\quad t=0,\dots,H-1.
\end{align}
with state $x_t\in\mathbb{R}^n$ and $u_t\in\mathbb{R}^m$. The disturbances $\{w_t\}$ are i.i.d., zero mean, with covariance $\Sigma_w\succ0$. Inputs are generated by a causal, possibly randomized, policy $\pi\in\Pi$, i.e., $u_t\sim\pi_t(x_{0:t})$, where the internal randomization is independent of future disturbances. Stack the regressors as
\[
z_t\coloneqq\begin{bmatrix}x_t^\top&u_t^\top\end{bmatrix}^\top
\in\mathbb{R}^d,\qquad d\coloneqq n+m,
\]
and define the covariance matrix induced by $\pi$ as
\[
\cov{\pi}\coloneqq
\mathbb{E}^{\pi}\!\left[\sum_{t=0}^{H-1}z_tz_t^\top\right]
\succeq0.
\]
Let $\mathcal{C}\coloneqq\{\cov{\pi}:\pi\in\Pi\}$ denote the set of achievable covariance matrices, and let
\begin{align}\label{eq:feas-set}
\mathcal{D}\coloneqq\bigl\{\Sigma\in\mathcal{C}:\ \Tr\bigl(W^{(2)}\Sigma\bigr)\le\beta\bigr\}
\end{align}
for a weight $W^{(2)}\succ0$ and budget $\beta>0$. We study
\begin{equation}\label{eq:synthesis}
\min_{\Sigma\in\mathcal{D}}\; f(\Sigma),
\qquad
f(\Sigma)\coloneqq\Tr\!\Big(W^{(1)}\bigl(\Sigma_0+\Sigma\bigr)^{-1}\Big),
\end{equation}
where $W^{(1)}\succ0$ weights the directions in which information is valuable and $\Sigma_0\succ0$ is a fixed regularization, which may encode information available prior to the experiment. The budget constraint prevents the optimizer from driving $\Sigma$ to infinity. It can equivalently be interpreted as a prescribed energy budget.
\subsection{Motivation}\label{motiv}
A natural motivation for problems of the form~\eqref{eq:synthesis} arises in system identification, where inputs are designed to make the collected data informative. Rather than injecting excitation naively, active experiment design shapes the information matrix in useful directions. Writing $\theta=[A\ B]$ and $x_{t+1}=\theta z_t+w_t$, the weighted error of the least-squares estimate $\widehat\theta$, based on data with information matrix $\Sigma_{\mathrm{data}}\coloneqq\sum_t z_tz_t^\top$, satisfies
\[
\mathbb{E}\|\widehat\theta-\theta\|_{W}^2
\approx\Tr\bigl(W\,\Sigma_{\mathrm{data}}^{-1}\bigr),
\]
where $W$ encodes the parameter-error directions relevant to the downstream task and may be derived from the task cost and noise covariance \cite{wagenmaker2021task,wagenmaker2020active}. Thus, designing next experiment to reduce the task-weighted estimation error amounts to solving \eqref{eq:synthesis} with $\Sigma_0$ equal to the information matrix of previously collected data.

We impose the following assumptions.
\begin{assumption}\label{ass:standing}
\begin{enumerate}
\item[\textit{(i)}] $\Sigma_w\succ0$, $W^{(1)}\succ0$, $W^{(2)}\succ0$, $\Sigma_0\succ0$.
\item[\textit{(ii)}] $\beta>b_0\coloneqq\min_{\pi\in\Pi}\Tr\bigl(W^{(2)}\cov{\pi}\bigr)$.
\end{enumerate}
\end{assumption}
Part \textit{(ii)} ensures that the budget constraint is strictly feasible. Note that $b_0$ is itself the optimal value of a standard finite-horizon LQ problem (minimize the expected $W^{(2)}$-weighted energy) and is therefore computable.

\begin{lemma}[Feasible set]\label{lem:feasible-set}
$\mathcal{D}$ is convex and compact. Moreover, every $\Sigma\in\mathcal{D}$ satisfies
\[
\lambda_{\max}(\Sigma)\le\Tr(\Sigma)\le\tfrac{\beta}{\mu},
\qquad \mu\coloneqq\lambda_{\min}(W^{(2)}),
\]
and consequently
\[
\operatorname{diam}_F(\mathcal{D})
\coloneqq\max_{\Sigma,\Sigma'\in\mathcal{D}}
\|\Sigma-\Sigma'\|_F
\le\tfrac{2\beta}{\mu}.
\]
\end{lemma}
\begin{proof}
	Let $\pi_1,\pi_2\in\Pi$ and $\alpha\in[0,1]$, and let $\pi_\alpha$ select $\pi_1$ with probability $\alpha$ and $\pi_2$ with probability $1-\alpha$ at $t=0$, then follow the selected policy throughout. Conditioning on this choice gives
	\[
	\cov{\pi_\alpha}=\alpha\cov{\pi_1}+(1-\alpha)\cov{\pi_2},
	\]
	so $\mathcal{C}$ is convex and $\mathcal{D}$ is its intersection with the half-space $\{\Sigma:\Tr(W^{(2)}\Sigma)\le\beta\}$. Since this half-space is convex, $\mathcal{D}$ is convex as the intersection of two convex sets. For any $\Sigma\in\mathcal{D}$,
	\[
	\beta\ge\Tr(W^{(2)}\Sigma)
	\ge\mu\Tr(\Sigma)
	\ge\mu\lambda_{\max}(\Sigma),
	\]
	and, since $\Sigma\succeq0$,
	\[
	\|\Sigma\|_F\le\Tr(\Sigma)\le\tfrac{\beta}{\mu}.
	\]
	The diameter bound follows by the triangle inequality. Finally, $\mathcal{C}$ is closed by Lemma~\ref{lem:closed}, so $\mathcal{D}$ is closed; being also bounded in the finite-dimensional space of symmetric $d\times d$ matrices, it is compact.
\end{proof}

\begin{lemma}[Objective]\label{lem:objective}
The function $f$ in \eqref{eq:synthesis} is convex and differentiable on the positive semidefinite cone, with
\begin{align}\label{eq:gradient}
\nabla f(\Sigma)=-\,(\Sigma_0+\Sigma)^{-1}W^{(1)}(\Sigma_0+\Sigma)^{-1}\;\prec\;0 .
\end{align}
Its gradient is $L$-Lipschitz w.r.t. $\|\cdot\|_F$ on $\{\Sigma\succeq 0\}$ with
\begin{align}\label{eq:L-const}
L=\tfrac{2\|W^{(1)}\|_2}{\lambda_{\min}^3(\Sigma_0)},
\end{align}
and $f$ is $m$-strongly convex on $\mathcal{D}$ with respect to $\|\cdot\|_F$ with
\begin{align}\label{eq:m-const}
m=\tfrac{2\lambda_{\min}(W^{(1)})}{\bigl(\lambda_{\max}(\Sigma_0)+\tfrac{\beta}{\mu}\bigr)^3},
\qquad \mu=\lambda_{\min}(W^{(2)}).
\end{align}
\end{lemma}
\begin{proof}
	See Appendix~\ref{app:objective}.
\end{proof}

\section{Frank--Wolfe over achievable covariances}\label{sec:fw}

Projected gradient descent on \eqref{eq:synthesis} would require Euclidean projections onto $\mathcal{D}$, which is defined only implicitly through the dynamics and the policy class. The Frank--Wolfe method \cite{frank1956algorithm,jaggi2013revisiting} avoids projections. Starting from any $\Sigma^{(0)}\in\mathcal{D}$, at iteration $i$ it solves the \emph{linear minimization oracle} (LMO)
\begin{align}\label{eq:lmo}
S^{(i)}\in\argmin_{\Sigma\in\mathcal{D}}\ \Tr\bigl(M^{(i)}\Sigma\bigr),
\qquad
M^{(i)}\coloneqq\nabla f\bigl(\Sigma^{(i)}\bigr),
\end{align}
and updates, with step size $\alpha_i=\tfrac{2}{i+2}$,
\begin{align}\label{eq:fw-update}
\Sigma^{(i+1)}=(1-\alpha_i)\,\Sigma^{(i)}+\alpha_i\,S^{(i)} .
\end{align}
Feasibility is automatic because $\Sigma^{(i+1)}$ is a convex combination of feasible points.
The Frank--Wolfe gap is
\[
G_i^{\mathrm{FW}}
\coloneqq
\Tr\!\left(
M^{(i)}\bigl(\Sigma^{(i)}-S^{(i)}\bigr)
\right)
\geq 0.
\]
It upper bounds the suboptimality $f(\Sigma^{(i)})-f^\star$ and therefore serves as a stopping criterion for the method.

In our setting, the LMO will be solved inexactly (by bisection, Section~\ref{sec:lmo}), so we state the convergence guarantee for $\delta$-approximate oracles, in which each iteration returns $S^{(i)}\in\mathcal{D}$ satisfying
\begin{align}\label{eq:inexact-lmo}
\Tr\bigl(M^{(i)}S^{(i)}\bigr)\le\min_{\Sigma\in\mathcal{D}}\Tr\bigl(M^{(i)}\Sigma\bigr)+\delta .
\end{align}

\begin{theorem}[Convergence]\label{thm:rate}
Let Assumption~\ref{ass:standing} hold, let $\Sigma^\star$ be the (unique) minimizer of \eqref{eq:synthesis}, and let $\{\Sigma^{(i)}\}$ be generated by \eqref{eq:fw-update} with $\alpha_i=\tfrac2{i+2}$ and a $\delta$-approximate LMO \eqref{eq:inexact-lmo}. Then for all $M\ge1$,
\begin{align}\label{eq:obj-rate}
f\bigl(\Sigma^{(M)}\bigr)-f\bigl(\Sigma^\star\bigr)
\;\le\;
\tfrac{2C_f}{M+2}+\delta,
\quad
C_f\le\tfrac{8\,\|W^{(1)}\|_2\,\beta^2}{\lambda_{\min}^3(\Sigma_0)\,\mu^2},
\end{align}
where $C_f$ is the curvature constant of $f$ over $\mathcal{D}$ and $\mu=\lambda_{\min}(W^{(2)})$. Moreover, by strong convexity,
\begin{align}\label{eq:iter-rate}
\bigl\|\Sigma^{(M)}-\Sigma^\star\bigr\|_F
\le
\sqrt{\tfrac{2}{m}\Bigl(\tfrac{2C_f}{M+2}+\delta\Bigr)}
\end{align}
with $m$ as in \eqref{eq:m-const}. In particular, with exact oracles ($\delta=0$) the objective converges at rate $\mathcal{O}(1/M)$ and the iterates at rate $\mathcal{O}(1/\sqrt{M})$.
\end{theorem}
\begin{proof}
	See Appendix~\ref{app:rate}. Uniqueness and existence of $\Sigma^\star$ follow from strong convexity and compactness of $\mathcal{D}$.
\end{proof}
\begin{remark}[Realizing the iterates as a policy]\label{rem:mixture}
	Since $\alpha_0=1$, unrolling \eqref{eq:fw-update} yields
	\[
	\Sigma^{(M)}=\sum_{k=0}^{M-1}p_kS^{(k)},
	\qquad p_k=\tfrac{2(k+1)}{M(M+1)}.
	\]
	Each search point $S^{(k)}$ is induced by an explicit policy $\pi^{(k)}$ returned by the LMO. Hence, the randomized policy $\tilde{\pi}$ that, at $t=0$, draws $k\in\{0,\dots,M-1\}$ with probability $p_k$ and follows $\pi^{(k)}$ for the entire horizon satisfies
	\[
	\cov{\tilde{\pi}}
	=\sum_{k=0}^{M-1}p_k\cov{\pi^{(k)}}
	=\Sigma^{(M)}
	\]
	exactly, and therefore inherits the guarantee \eqref{eq:obj-rate}. See also \cite{jaggi2013revisiting,mutny2023active}.
\end{remark}
\section{The linear subproblem is constrained LQ}\label{sec:lmo}

Fix a gradient matrix $M\coloneqq M^{(i)}\prec0$ (cf.~\eqref{eq:gradient}) and consider the LMO \eqref{eq:lmo}. Using
\[
\Tr(M\cov{\pi})
=\mathbb{E}^\pi\left[\sum_{t=0}^{H-1}z_t^\top Mz_t\right],
\]
the LMO is the optimal control problem
\begin{equation}\label{eq:lmo-control}
p^\star=\min_{\pi\in\Pi}\ \mathbb{E}^{\pi}\!\Big[\sum_{t=0}^{H-1}z_t^\top M z_t\Big]
\quad\text{s.t.}\quad
\Tr\bigl(W^{(2)}\cov{\pi}\bigr)\le\beta,
\end{equation}
i.e., a finite-horizon LQ problem with a negative definite stage cost and a single scalar quadratic constraint. We solve it by dualizing the constraint. For $\lambda\ge0$, define the stage weight $\widetilde M(\lambda)\coloneqq M+\lambda W^{(2)}$ and
\begin{multline}
	g(\lambda)\coloneqq-\lambda\beta+v(\lambda),\;
	v(\lambda)\coloneqq\inf_{\pi\in\Pi}\mathbb{E}^\pi\!\Big[\sum_{t=0}^{H-1}z_t^\top\widetilde M(\lambda)z_t\Big].
	\label{eq:dual-fn}
\end{multline}
The inner problem in \eqref{eq:dual-fn} is an unconstrained generalized LQ problem. Partition $\widetilde M(\lambda)$ conformally with $(x,u)$ as $\widetilde M^{xx}$, $\widetilde M^{xu}$, $\widetilde M^{ux}$, and $\widetilde M^{uu}$. With $P_H(\lambda)=0$, the optimal policy is obtained from the backward Riccati recursion
\begin{equation}\label{eq:riccati}
\begin{aligned}
S_t(\lambda)&=\widetilde M^{uu}(\lambda)
 +B^\top P_{t+1}(\lambda)B,\\
K_t(\lambda)&=S_t(\lambda)^{-1}
\bigl(\widetilde M^{ux}(\lambda)+B^\top P_{t+1}(\lambda)A\bigr),\\
P_t(\lambda)&=\widetilde M^{xx}(\lambda)
 +A^\top P_{t+1}(\lambda)A\\
&\quad-\bigl(\widetilde M^{ux}(\lambda)
 +B^\top P_{t+1}(\lambda)A\bigr)^\top K_t(\lambda).
\end{aligned}
\end{equation}
We say $\lambda$ is \emph{admissible} if $S_t(\lambda)\succ0$ for all $t$, and write $\Lambda\subseteq[0,\infty)$ for the set of admissible $\lambda$. For admissible $\lambda$, a standard dynamic programming argument shows that the unique optimal policy is the linear feedback $u_t=-K_t(\lambda)x_t$, denoted $\pi_\lambda$, with value
\[
v(\lambda)=\sum_{t=0}^{H-1}
\Tr\bigl(P_{t+1}(\lambda)\Sigma_w\bigr).
\]
Its covariance $\cov{\pi_\lambda}$, and hence the \emph{budget map}
\begin{align}\label{eq:budget-map}
b(\lambda)\coloneqq\Tr\bigl(W^{(2)}\cov{\pi_\lambda}\bigr),
\end{align}
is obtained by propagating the closed-loop second moments: with $\Sigma_{x,0}=0$ and $A_t\coloneqq A-BK_t(\lambda)$,
\[
\Sigma_{x,t+1}=A_t\Sigma_{x,t}A_t^\top+\Sigma_w,\qquad
\cov{\pi_\lambda}=\sum_{t=0}^{H-1}
\begin{bmatrix} I\\ -K_t \end{bmatrix}
\Sigma_{x,t}
\begin{bmatrix} I\\ -K_t \end{bmatrix}^{\!\top}.
\]
The next lemma collects the properties that lead to a principled bisection procedure for $\lambda$.
\begin{lemma}[Structure of the dual family]\label{lem:subproblem}
Let $M\prec0$, $\mu=\lambda_{\min}(W^{(2)})$, and let Assumption~\ref{ass:standing} hold. Then:
\begin{enumerate}
\item[\textit{(i)}] With $\lambda_{\mathrm{crit}}\coloneqq\inf\Lambda$, we have
$\Lambda=(\lambda_{\mathrm{crit}},\infty)$ with
$0<\lambda_{\mathrm{crit}}\le\|M\|_2/\mu$. In particular, every $\lambda>\|M\|_2/\mu$ is admissible.
\item[\textit{(ii)}] $b(\cdot)$ is continuous and nonincreasing on $\Lambda$.
\item[\textit{(iii)}] (Certificate) For every admissible $\lambda$ with $b(\lambda)\le\beta$, the policy $\pi_\lambda$ is feasible for \eqref{eq:lmo-control} and
\[
\Tr\bigl(M\cov{\pi_\lambda}\bigr)
\le p^\star+\lambda\bigl(\beta-b(\lambda)\bigr).
\]
In particular, if $b(\lambda)=\beta$ then $\pi_\lambda$ is exactly optimal.
\item[\textit{(iv)}] (Bracket) With $b_0$ as defined in Assumption~\ref{ass:standing}, $b(\lambda)\le\beta$ for all
\[
\lambda\ge\bar{\lambda}
\coloneqq
\tfrac{\beta}{\beta-b_0}\,\tfrac{\|M\|_2}{\mu}.
\]
\end{enumerate}
\end{lemma}
\begin{proof}
	See Appendix~\ref{app:subproblem}.
\end{proof}
Lemma~\ref{lem:subproblem} yields the following LMO procedure. Initialize $[\lambda_{\mathrm{lo}},\lambda_{\mathrm{hi}}]=[0,\bar\lambda]$. By \textit{(i)} and \textit{(iv)}, $\lambda_{\mathrm{hi}}$ is admissible and satisfies $b(\lambda_{\mathrm{hi}})\le\beta$. At each step, evaluate the midpoint $\lambda$. If \eqref{eq:riccati} fails to satisfy $S_t(\lambda)\succ0$ for some $t$, or if $b(\lambda)>\beta$, set $\lambda_{\mathrm{lo}}\leftarrow\lambda$; otherwise, set $\lambda_{\mathrm{hi}}\leftarrow\lambda$. By \textit{(ii)}, the upper endpoint remains admissible with $b(\lambda_{\mathrm{hi}})\le\beta$, while \textit{(iii)} bounds the LMO error of $\pi_{\lambda_{\mathrm{hi}}}$ by $\lambda_{\mathrm{hi}}(\beta-b(\lambda_{\mathrm{hi}}))$. Stopping when this certificate is at most $\delta$ yields a $\delta$-approximate oracle in the sense of \eqref{eq:inexact-lmo}, as required by Theorem~\ref{thm:rate}. If $b(\lambda)=\beta$ has a solution in $\Lambda$, continuity and monotonicity ensure that the certificate converges to zero as the bracket shrinks.
\begin{remark}[Active budget and the degenerate case]\label{rem:degenerate}
Since $M\prec0$, every solution of \eqref{eq:lmo-control} exhausts the budget.
Convexity of $\mathcal C$ and strict feasibility of the budget constraint, guaranteed by Assumption~\ref{ass:standing}\textit{(ii)}, imply strong duality. At $\lambda=0$, the unconstrained negative-definite quadratic cost is unbounded below, so $g(0)=-\infty$. Hence every optimal multiplier satisfies $\lambda^\star>0$.
Complementary slackness therefore gives $\Tr(W^{(2)}\cov{\pi^\star})=\beta$. Typically, $b(\lambda)\to\infty$ as
$\lambda\downarrow\lambda_{\mathrm{crit}}$, where $\downarrow$ denotes
convergence from above. Hence, bisection finds
$\lambda^\star\in\Lambda$ with $b(\lambda^\star)=\beta$. If instead $b(\lambda)<\beta$ for all $\lambda\in\Lambda$, some $S_t(\lambda_{\mathrm{crit}})$ must be singular.
Indeed, each $S_t(\lambda)$ is continuous in $\lambda$. If every $S_t(\lambda_{\mathrm{crit}})$ were positive definite, the Riccati recursion would remain well posed for some $\lambda<\lambda_{\mathrm{crit}}$, contradicting the definition of $\lambda_{\mathrm{crit}}$.
The resulting null directions leave the fixed-$\lambda_{\mathrm{crit}}$ cost unchanged, so adding independent noise there increases the budget continuously.
Specifically, one may use $u_t=-K_tx_t+\eta_t$, where $\eta_t$ is independent noise supported on the null space of $S_t(\lambda_{\mathrm{crit}})$. This leaves the fixed-$\lambda_{\mathrm{crit}}$ cost unchanged, and the covariance of $\eta_t$ can be chosen so that $\Tr(W^{(2)}\cov{\pi})=\beta$.
The resulting policy is an exact solution by Lemma~\ref{lem:subproblem}\textit{(iii)}.
\end{remark}
\begin{remark}[Cost per iteration]
Each bisection step costs one Riccati recursion and one covariance propagation, i.e., $\mathcal{O}(Hd^3)$ arithmetic. The number of bisection steps to reach tolerance is logarithmic in $\bar\lambda/\delta$, so the overall method is computationally efficient compared to semidefinite-programming reformulations, and it scales to long horizons.
\end{remark}
\section{Extensions and applications}\label{sec:extensions}
The simplified problem \eqref{eq:synthesis} extends in several directions while preserving the structure of the Frank--Wolfe subproblems. In each case, the LMO remains a generalized LQ problem with a scalar bisection.
\subsection{Added LQR cost}
A natural variant includes a control cost penalty:
\begin{equation}\label{eq:synthesis-lqr}
	\min_{\Sigma\in\mathcal{D}}\;
	\Tr\!\Big(W^{(1)}(\Sigma_0+\Sigma)^{-1}\Big)
	+\Tr\bigl(W^{(0)}\Sigma\bigr).
\end{equation}
with $W^{(0)}\succeq0$ (e.g., $W^{(0)}=\operatorname{blkdiag}(Q,R)$). The extra term is linear in $\Sigma$, so Lemma~\ref{lem:objective} holds with the same constants ($L$ and $m$ are unaffected), and the only change to the method is that the LMO gradient becomes $M^{(i)}+W^{(0)}$, making the stage weight in \eqref{eq:riccati} equal to $M^{(i)}+W^{(0)}+\lambda W^{(2)}$. The budget constraint is retained to keep $\mathcal{D}$ compact. For large $\beta$ it is inactive, and \eqref{eq:synthesis-lqr} is effectively unconstrained. For $W^{(2)}=W^{(0)}$, for example, one could choose $\beta$ as the energy
$\Tr\bigl(W^{(2)}\Sigma^{\pi}\bigr)$ attained by LQR controller with additive probing noise, which provides a loose upper bound on the objective.

\subsection{Covariances over multiple horizons}\label{sec:multihorizon}
Let $0<\tau_1<\dots<\tau_K=H$ and
\[
\Sigma_j^{\pi}\coloneqq
\mathbb{E}^\pi\!\left[\sum_{t=0}^{\tau_j-1}z_tz_t^\top\right].
\]
The multi-horizon objective
\begin{equation}\label{eq:synthesis-mixture}
\min_{\pi\in\Pi}\ \sum_{j=1}^{K}\Tr\!\Big(W_j\bigl(\Sigma_0+\Sigma_{j}^{\pi}\bigr)^{-1}\Big)
\ \ \text{s.t.}\ \Tr\bigl(W^{(2)}\Sigma_{K}^{\pi}\bigr)\le\beta
\end{equation}
promotes informativeness at intermediate times as well as at the end of the experiment. Frank--Wolfe now runs over the tuple $(\Sigma_1^\pi,\dots,\Sigma_K^\pi)$, whose achievable set is convex by the same mixing argument. The linearized objective is
\[
\sum_j\Tr\bigl(M_j^{(i)}\Sigma_j^{\pi}\bigr)
=\mathbb{E}^\pi\!\left[\sum_t z_t^\top\widehat M_t^{(i)}z_t\right],
\]
with $W_j\succ0$ and
\[
M_j^{(i)}=-(\Sigma_0+\Sigma_j^{(i)})^{-1}W_j
(\Sigma_0+\Sigma_j^{(i)})^{-1},\qquad
\widehat M_t^{(i)}=\sum_{j:\,t<\tau_j}M_j^{(i)}.
\]
Since the Riccati recursion \eqref{eq:riccati} accommodates time-varying weights without modification, the LMO is again a generalized LQ problem, now with time-dependent stage weight $\widehat M_t^{(i)}+\lambda W^{(2)}$, with $\lambda$ found by scalar bisection as before. The analysis of Theorem~\ref{thm:rate} applies to the sum objective on the product feasible set, with curvature bounded by the sum of the per-component bounds.
\subsection{Input design for system identification}\label{sec:sysid}
Recall from~\ref{motiv} that, for system identification, \eqref{eq:synthesis} can be used to reduce the task-weighted estimation error by taking $\Sigma_0$ as the information matrix of previously collected data. Since $(A,B)$ are unknown, we proceed episodically using certainty equivalence. At episode $k$, we form the least-squares estimate
\[
\widehat\theta^{(k)}
=\bar\Sigma^{(k)}\bigl(\Sigma^{(k)}\bigr)^{-1},
\quad
\bar\Sigma^{(k)}\coloneqq\sum x_{t+1}z_t^\top,
\quad
\Sigma^{(k)}\coloneqq\sum z_tz_t^\top,
\]
where $\Sigma^{(k)}$ and $\bar\Sigma^{(k)}$ aggregate all data collected so far. We then solve \eqref{eq:synthesis} using the estimated dynamics and $\Sigma_0=\Sigma^{(k)}$ for $M$ Frank--Wolfe iterations, obtaining policies $\pi^{(k,0)},\dots,\pi^{(k,M-1)}$. Following Remark~\ref{rem:mixture}, we sample $\pi^{(k,j)}$ with probability $p_j=\tfrac{2(j+1)}{M(M+1)}$, execute it on true system for $H$ steps, update $(\Sigma^{(k)},\bar\Sigma^{(k)})$, and re-estimate the dynamics. Unlike \cite{wagenmaker2020active,wagenmaker2021task}, which optimize over restricted classes such as periodic inputs, our method optimizes directly over covariances achievable by causal feedback policies. Related Frank--Wolfe approaches to experiment design appear in \cite{mutny2023active,wagenmaker2024optimal}.
\subsection{Adaptive online LQR}\label{sec:adaptive}
Finally, the multi-horizon variant provides a computationally tractable approach to dual control. Naive exploration schemes \cite{simchowitz2020naive} add tuned random noise to a certainty equivalent controller, an \emph{explicit} dual control strategy \cite{aastrom2013adaptive,wittenmark1995adaptive}, while exact \emph{implicit} formulations via hyperstates are intractable \cite{rosdahl2022dual}. Related in spirit to our approach, the intrinsic-reward LQR algorithm of \cite{bartos2026optimistic} promotes uncertainty-driven exploration by augmenting the certainty equivalent synthesis cost while retaining the structure of a standard LQR problem. The following control oriented experiment design lies between explicit and implicit dual control. At each update time $\tau_k$, the learner maintains a posterior
$
\mathcal{N}\bigl(\operatorname{vec}(\hat\theta_k),
\,\Lambda_k^{-1}\otimes\Sigma_w\bigr)
$
over the vectorized parameters $\theta=[A\ B]$, where $\Lambda_k$ is the regularized information matrix formed from the data collected so far. It then minimizes the certainty equivalent control cost plus a prediction of the excess cost that future certainty equivalent controllers will incur due to estimation error:
\begin{align}
	\min_{\pi}\quad
	&\mathbb{E}^{\pi}_{\hat{\theta}_k}\!\Bigg[
	\sum_{t=\tau_k}^{T}
	\bigl(x_t^\top Qx_t+u_t^\top Ru_t\bigr)
	\Bigg]
	\label{eq:dual-control}\\
	&+\tfrac{1}{2}\sum_{m=k+1}^{K}(\tau_{m+1}-\tau_m)
	\Tr\!\Big(
	H(\hat\theta_k)
	\Bigl((\Lambda_k+\Sigma_m^\pi)^{-1}\!\otimes\Sigma_w\Bigr)
	\Big),
	\nonumber
\end{align}
where $\tau_{K+1}\coloneqq T$. The expectation is taken under the estimated dynamics $\hat\theta_k$ with the initial state fixed to the current state $x_{\tau_k}$, and
\[
\Sigma_m^{\pi}\coloneqq
\mathbb{E}^{\pi}_{\hat\theta_k}\!\left[
\sum_{t=\tau_k}^{\tau_m-1}z_tz_t^\top\right]
\]
is the covariance accumulated before the model is re-estimated at time $\tau_m$. The matrix $H(\hat\theta_k)\in\mathbb{R}^{dn\times dn}$ is the model-task Hessian, the Hessian of the certainty equivalent control cost with respect to the vectorized model parameters \cite{wagenmaker2021task}. Term $m$ in \eqref{eq:dual-control} predicts, via a second-order expansion of the control cost in the model parameters, the excess cost of the certainty equivalent controller that will be synthesized at time $\tau_m$ and deployed until $\tau_{m+1}$, since $(\Lambda_k+\Sigma_m^\pi)^{-1}\otimes\Sigma_w$ approximates the parameter error covariance at that point.

To connect \eqref{eq:dual-control} with the preceding sections, view $H(\hat\theta_k)$ as a $d\times d$ array of $n\times n$ blocks $[H]_{(i,j)}$, and define the partial trace $\Tr_{\Sigma_w}(H)\in\mathbb{R}^{d\times d}$ entrywise by
$
[\Tr_{\Sigma_w}(H)]_{ij}
\coloneqq\Tr\bigl(\Sigma_w[H]_{(i,j)}\bigr).
$
A direct computation shows that, for any symmetric $P$,
\[
\Tr\bigl(H(P\otimes\Sigma_w)\bigr)
=\Tr\bigl(\Tr_{\Sigma_w}(H)\,P\bigr).
\]
Thus, with
$
W_m\coloneqq\tfrac{1}{2}(\tau_{m+1}-\tau_m)
\Tr_{\Sigma_w}\bigl(H(\hat\theta_k)\bigr),
$
the penalty terms in \eqref{eq:dual-control} become $\Tr\bigl(W_m(\Lambda_k+\Sigma_m^\pi)^{-1}\bigr)$. Problem \eqref{eq:dual-control} is therefore the multi-horizon objective \eqref{eq:synthesis-mixture} with $\Sigma_0=\Lambda_k$, combined with an LQR term as in \eqref{eq:synthesis-lqr}. Retaining a budget constraint over the remaining horizon as in the preceding sections, the Frank--Wolfe machinery applies verbatim. The LMO at iteration $i$ is a generalized LQ problem with time-varying stage weights
\[
\widetilde M_t^{(i)}(\lambda)=\operatorname{blkdiag}(Q,R)+\sum_{m:\,t<\tau_m}M_m^{(i)}+\lambda W^{(2)},
\]
where $M_m^{(i)}=-(\Lambda_k+\Sigma_m^{(i)})^{-1}W_m(\Lambda_k+\Sigma_m^{(i)})^{-1}$ as in Section~\ref{sec:multihorizon}. The problem is re-solved at each update time in receding-horizon fashion, using the newly collected data to update $\hat\theta$ and $\Lambda$. Compared to naive exploration, the probing energy is placed only in directions that matter for control, which can enable lower regret when the system structure permits it.

\begin{figure*}[!t]
    \centering
    \begin{subfigure}[t]{0.32\textwidth}
        \centering
        \resizebox{\linewidth}{!}{\input{fw_objective_vs_beta.tex}}
        \caption{Covariance design.}
        \label{fig:fw-objective-vs-beta}
    \end{subfigure}\hfill
    \begin{subfigure}[t]{0.32\textwidth}
        \centering
        \resizebox{\linewidth}{!}{\input{system_id_error_mc.tex}}
        \caption{Identification error.}
        \label{fig:system-id-error}
    \end{subfigure}\hfill
    \begin{subfigure}[t]{0.32\textwidth}
        \centering
        \resizebox{\linewidth}{!}{\input{system_id_cost_mc.tex}}
        \caption{Design objective.}
        \label{fig:system-id-cost}
    \end{subfigure}
    \caption{Numerical results. (a) Objective value versus Frank--Wolfe iteration for several budgets $\beta$. (b) Identification error $\|\hat A^{(j)}-A\|_F+\|\hat B^{(j)}-B\|_F$ versus episode for Frank--Wolfe, certainty equivalence, naive exploration, and the periodic-input baseline. (c) Design objective versus episode for the same four methods.}
    \label{fig:numerics}
\end{figure*}
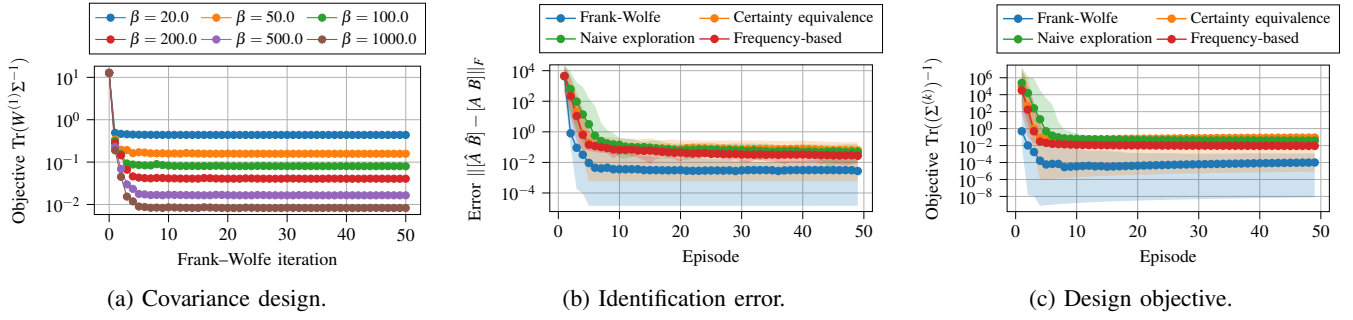
\section{Numerical examples}\label{sec:numerics}
We illustrate the approach on three tasks: covariance design with known dynamics, input design for system identification, and adaptive online LQR.\footnote{For implementation details and code, see \url{https://github.com/Fethi-Bencherki/Optimal-input-design-via-FW}.}
\subsection{Covariance design}
We first solve \eqref{eq:synthesis} for the known system
\[
A=\begin{bmatrix}0.9&1.0\\0.0&0.9\end{bmatrix},\quad
B=\begin{bmatrix}0\\1.0\end{bmatrix},\quad
\Sigma_w=0.02\,I_2,
\]
with $H=100$, $W^{(1)}=\operatorname{diag}(2.0,1.0,0.5)$, and $W^{(2)}=\operatorname{diag}(1.0,1.0,0.3)$. We initialize with a feasible covariance from the fixed-$\lambda$ LQ subproblem at $\lambda=1$, run the inner bisection to high accuracy, and perform $50$ Frank--Wolfe iterations for several budgets $\beta$. Figure~\ref{fig:fw-objective-vs-beta} shows monotone convergence, consistent with the $\mathcal{O}(1/M)$ rate of Theorem~\ref{thm:rate}; larger budgets yield more informative covariances and lower limiting values. Every iterate is feasible by construction.
\subsection{Experiment design for system identification}
Next we run the episodic scheme of Section~\ref{sec:sysid} on the same system with $(A,B)$ unknown. We compare against three baselines, all normalized to the same energy budget per episode\footnote{The budget constraint is enforced under the estimated dynamics and may therefore be violated under the true dynamics. If it is a hard constraint, it can be tightened by a margin that accounts for model uncertainty.}:
\begin{enumerate}
\item[\textit{(i)}] \emph{Certainty equivalence}, which rolls out the LQ policy for the current estimate without seeking exploration;
\item[\textit{(ii)}] \emph{Naive exploration}, which applies random inputs; and
\item[\textit{(iii)}] a frequency-search baseline based on \cite{wagenmaker2020active}, which selects the best periodic input for the current model estimate.
\end{enumerate}
At each episode we record the estimation error $\|\hat A^{(j)}-A\|_F+\|\hat B^{(j)}-B\|_F$. Figures~\ref{fig:system-id-error} and \ref{fig:system-id-cost} show the estimation error and design objective across Monte Carlo noise realizations.
\subsection{Adaptive online LQR}
For the adaptive online LQR experiment, we use
\[
A=\begin{bmatrix}1.2&1.0\\0&1.0\end{bmatrix},\qquad
B=\begin{bmatrix}0\\1\end{bmatrix},\qquad
\Sigma_w=0.09\,I_2.
\]
We apply the receding-horizon scheme of Section~\ref{sec:adaptive}, using the identity as a surrogate for the model-task Hessian $H(\hat\theta)$.

\begin{figure}[t]
    \centering
    \resizebox{\columnwidth}{!}{\input{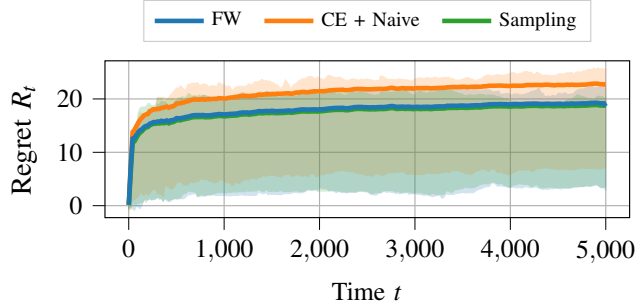}}
    \caption{Cumulative regret in adaptive online LQR for the proposed method, certainty-equivalent LQR with naive exploration, and a sampling-based baseline.}
    \label{fig:adaptive-online-lqr-regret}
\end{figure}

Performance is measured by cumulative regret $R_T^\pi=C_T^\pi-TJ^\star$, where $C_T^\pi$ is the cumulative LQR cost and $J^\star$ is the optimal steady-state cost under the true dynamics. We compare against certainty-equivalent LQR with naive exploration noise and a sampling-based optimization using a large number of samples, which locally perturbs the certainty-equivalent gains and selects the best candidate under \eqref{eq:dual-control}. The latter serves as a reference to verify that the Frank--Wolfe procedure achieves similar behavior. Figure~\ref{fig:adaptive-online-lqr-regret} reports regret over 50 noise realizations. The Frank--Wolfe controller directs its exploration effort and achieves lower regret than the naive approach.
\section{Conclusion}
We presented a Frank--Wolfe approach to budget-constrained covariance design for linear systems. The method operates directly on covariance matrices achievable by causal policies, with each linear minimization subproblem reduced to a finite-horizon LQ problem solved by Riccati recursion and scalar bisection, yielding an $\mathcal{O}(1/M)$ convergence rate. The same machinery applies to system identification with unknown dynamics and to an adaptive LQR scheme interpolating between explicit and implicit dual control. Future work includes regret guarantees for the adaptive scheme and infinite-horizon formulations.
\section*{APPENDIX}

\subsection{Closedness of the achievable covariances set}\label{app:closed}

\begin{lemma}\label{lem:closed}
The set $\mathcal{C}$ of achievable covariance matrices is closed.
\end{lemma}
\begin{proof}
	For $\pi\in\Pi$, let $\Sigma_{z,t}\coloneqq\mathbb{E}^\pi[z_tz_t^\top]$, and let $[\cdot]_{xx},[\cdot]_{xu},[\cdot]_{ux},[\cdot]_{uu}$ denote the blocks conformal with $(x,u)$. Every such sequence satisfies
\begin{itemize}
	\item[\textit{(a)}] $\Sigma_{z,t}\succeq0$ for $t=0,\dots,H-1$;
	\item[\textit{(b)}] $[\Sigma_{z,0}]_{xx}=0$;
	\item[\textit{(c)}] $[\Sigma_{z,t+1}]_{xx}=[A\ B]\Sigma_{z,t}[A\ B]^\top+\Sigma_w$ for $t=0,\dots,H-2$.
\end{itemize}
	Indeed, \textit{(b)} follows from $x_0=0$, while \textit{(c)} follows from \eqref{eq:system} because $w_t$ is zero mean and independent of $z_t$.
	
	Conversely, let $(\Sigma_{z,t})_{t=0}^{H-1}$ satisfy \textit{(a)}--\textit{(c)}, set $X_t\coloneqq[\Sigma_{z,t}]_{xx}$, and consider $u_t=K_tx_t+\eta_t$, with
\[
K_t\coloneqq[\Sigma_{z,t}]_{ux}X_t^+,
\qquad
\Theta_t\coloneqq[\Sigma_{z,t}]_{uu}
-[\Sigma_{z,t}]_{ux}X_t^+[\Sigma_{z,t}]_{xu},
\]
where $\eta_t$ is zero mean with covariance $\Theta_t$, independent across time and of the disturbances. By the generalized Schur complement, $\Sigma_{z,t}\succeq0$ implies $\operatorname{range}([\Sigma_{z,t}]_{xu})\subseteq\operatorname{range}(X_t)$ and $\Theta_t\succeq0$,
	so the policy is well defined. An induction shows that $\mathbb{E}[z_tz_t^\top]=\Sigma_{z,t}$. Indeed, if $\mathbb{E}[x_tx_t^\top]=X_t$, which holds at $t=0$ by \textit{(b)}, then
\begin{align*}
\mathbb{E}[x_tu_t^\top]
=X_tK_t^\top=[\Sigma_{z,t}]_{xu},\,
\mathbb{E}[u_tu_t^\top]
=K_tX_tK_t^\top+\Theta_t=[\Sigma_{z,t}]_{uu}.
\end{align*}
	where the range inclusion is used in both identities; hence $\mathbb{E}[z_tz_t^\top]=\Sigma_{z,t}$, and \textit{(c)} yields $\mathbb{E}[x_{t+1}x_{t+1}^\top]=X_{t+1}$.
	Therefore,
	\[
	\mathcal{C}=\left\{\sum_{t=0}^{H-1}\Sigma_{z,t}:(\Sigma_{z,t})_t\in\mathcal{T}\right\},
	\]
	where $\mathcal{T}$ denotes the set of sequences satisfying \textit{(a)}--\textit{(c)}. Let $\Sigma^{(k)}\to\bar\Sigma$ with $\Sigma^{(k)}\in\mathcal{C}$, realized by $(\Sigma_{z,t}^{(k)})_t\in\mathcal{T}$. Since $0\preceq\Sigma_{z,t}^{(k)}\preceq\sum_{s=0}^{H-1}\Sigma_{z,s}^{(k)}=\Sigma^{(k)}$,
	each sequence $(\Sigma_{z,t}^{(k)})_k$ is bounded. As $H$ is finite, a common subsequence satisfies $\Sigma_{z,t}^{(k)}\to\bar\Sigma_{z,t}$ for all $t$. The constraints \textit{(a)}--\textit{(c)} are preserved under limits, so $(\bar\Sigma_{z,t})_t\in\mathcal{T}$ and
	$
	\bar\Sigma=\sum_{t=0}^{H-1}\bar\Sigma_{z,t}\in\mathcal{C}.
	$
	Thus, $\mathcal{C}$ is closed.
\end{proof}
\subsection{Proof of Lemma~\ref{lem:objective}}\label{app:objective}
\emph{Convexity.} Throughout, $\langle X,Y\rangle=\Tr(X^\top Y)$. Convexity of $f$ follows since $X\mapsto X^{-1}$ is matrix convex on the positive definite cone, $\Sigma\mapsto\Sigma_0+\Sigma$ is affine, and $\Tr(W^{(1)}\cdot)$ is linear and monotone. The gradient formula \eqref{eq:gradient} is standard. Since $\Sigma_0\succ0$ and $\Sigma\succeq0$, we have $Y\coloneqq(\Sigma_0+\Sigma)^{-1}\succ0$. Thus, for every $v\neq0$,
$
v^\top YW^{(1)}Yv=(Yv)^\top W^{(1)}(Yv)>0,
$
because $W^{(1)}\succ0$ and $Yv\neq0$. Hence $YW^{(1)}Y\succ0$, and therefore $\nabla f(\Sigma)=-YW^{(1)}Y\prec0$.

\emph{Smoothness.} Fix $\Sigma,\Sigma'\succeq0$, let $\Delta\coloneqq\Sigma'-\Sigma$, and define
\[
Y_\varepsilon\coloneqq(\Sigma_0+\Sigma+\varepsilon\Delta)^{-1},
\quad
G(\varepsilon)\coloneqq\nabla f(\Sigma+\varepsilon\Delta)
=-Y_\varepsilon W^{(1)}Y_\varepsilon .
\]
Since $\Sigma+\varepsilon\Delta=(1-\varepsilon)\Sigma+\varepsilon\Sigma'\succeq0$,
$\|Y_\varepsilon\|_2\le\lambda_{\min}(\Sigma_0)^{-1}$. Moreover,
$Y_\varepsilon'=-Y_\varepsilon\Delta Y_\varepsilon$, so
\[
G'(\varepsilon)
=
Y_\varepsilon\Delta Y_\varepsilon W^{(1)}Y_\varepsilon
+
Y_\varepsilon W^{(1)}Y_\varepsilon\Delta Y_\varepsilon,
\]
Using $\|XYZ\|_F\le\|X\|_2\|Y\|_F\|Z\|_2$, we obtain
\[
\|G'(\varepsilon)\|_F
\le
2\|Y_\varepsilon\|_2^3\|W^{(1)}\|_2\|\Delta\|_F
\le
L\|\Delta\|_F.
\]
Therefore,
\[
\|\nabla f(\Sigma')-\nabla f(\Sigma)\|_F
\le
\int_0^1\|G'(\varepsilon)\|_F\,\mathrm{d}\varepsilon
\le
L\|\Sigma'-\Sigma\|_F,
\]
where $L$ is given by \eqref{eq:L-const}.

\emph{Strong convexity.} It suffices to show that, for every $\Sigma\in\mathcal{D}$ and symmetric $\Delta$, $\nabla^2 f(\Sigma)[\Delta,\Delta]\coloneqq\left.\tfrac{\mathrm{d}^2}{\mathrm{d}t^2}f(\Sigma+t\Delta)\right|_{t=0}\ge m\|\Delta\|_F^2$. Let $Y\coloneqq(\Sigma_0+\Sigma)^{-1}$. Direct differentiation gives $\nabla^2f(\Sigma)[\Delta,\Delta]=2\Tr\bigl(W^{(1)}Y\Delta Y\Delta Y\bigr)$.
Since
\[
Y\Delta Y\Delta Y
=Y^{1/2}\bigl(Y^{1/2}\Delta Y^{1/2}\bigr)^2Y^{1/2}
\succeq0
\]
and $W^{(1)}\succeq\lambda_{\min}(W^{(1)})I$, we obtain
\begin{align*}
\Tr\bigl(W^{(1)}Y\Delta Y\Delta Y\bigr)
&\ge\lambda_{\min}(W^{(1)})\Tr\bigl(Y\Delta Y\Delta Y\bigr)\\
&\ge\lambda_{\min}(W^{(1)})\lambda_{\min}(Y)
\bigl\|Y^{1/2}\Delta Y^{1/2}\bigr\|_F^2\\
&\ge\lambda_{\min}(W^{(1)})\lambda_{\min}(Y)^3
\|\Delta\|_F^2.
\end{align*}
The final inequality follows from
$\|Y^{1/2}\Delta Y^{1/2}\|_F
\ge\lambda_{\min}(Y)\|\Delta\|_F$.
By Lemma~\ref{lem:feasible-set}, $\lambda_{\min}(Y)\ge\bigl(\lambda_{\max}(\Sigma_0)+\tfrac{\beta}{\mu}\bigr)^{-1}$, yielding \eqref{eq:m-const}. \hfill$\blacksquare$
\subsection{Proof of Theorem~\ref{thm:rate}}\label{app:rate}
Recall that
\[
C_f\coloneqq\sup_{\substack{\Sigma,S\in\mathcal{D}\\ \gamma\in(0,1]}}
\tfrac{2}{\gamma^2}\!\left[f\bigl(\Sigma+\gamma(S-\Sigma)\bigr)-f(\Sigma)-\gamma\langle\nabla f(\Sigma),S-\Sigma\rangle\right].
\]
Since $\nabla f$ is $L$-Lipschitz, the descent lemma gives
\[
f\bigl(\Sigma+\gamma(S-\Sigma)\bigr)\le f(\Sigma)+\gamma\langle\nabla f(\Sigma),S-\Sigma\rangle+\tfrac{\gamma^2L}{2}\|S-\Sigma\|_F^2.
\]
Thus, by Lemma~\ref{lem:feasible-set} and \eqref{eq:L-const},
\[
C_f\le L\operatorname{diam}_F(\mathcal{D})^2\le L\left(\tfrac{2\beta}{\mu}\right)^2
=\tfrac{8\|W^{(1)}\|_2\beta^2}{\lambda_{\min}^3(\Sigma_0)\mu^2}.
\]
Let $h_i\coloneqq f(\Sigma^{(i)})-f(\Sigma^\star)$ and $M^{(i)}\coloneqq\nabla f(\Sigma^{(i)})$. By the curvature bound, the $\delta$-approximate LMO, and convexity,
\begin{multline*}
	h_{i+1}
	\le h_i+\alpha_i\langle M^{(i)},S^{(i)}-\Sigma^{(i)}\rangle+\tfrac{\alpha_i^2C_f}{2}
	\le h_i\\+\alpha_i\bigl(\langle M^{(i)},\Sigma^\star-\Sigma^{(i)}\rangle+\delta\bigr)+\tfrac{\alpha_i^2C_f}{2}
	\le(1-\alpha_i)h_i+\alpha_i\delta+\tfrac{\alpha_i^2C_f}{2},
\end{multline*}
where $\langle M^{(i)},\Sigma^\star-\Sigma^{(i)}\rangle\le-h_i$. We prove by induction that $h_i\le 2C_f/(i+2)+\delta$ for all $i\ge1$. Since $\alpha_0=1$, $h_1\le\delta+C_f/2\le\delta+2C_f/3$. Now, assume $h_i\le 2C_f/(i+2)+\delta$ for some $i\ge1$. Substituting this bound and $\alpha_i=2/(i+2)$ into the recursion gives
\begin{align*}
h_{i+1}
&\le\left(1-\tfrac{2}{i+2}\right)
\left(\tfrac{2C_f}{i+2}+\delta\right)
+\tfrac{2\delta}{i+2}+\tfrac{2C_f}{(i+2)^2}\\
&=\tfrac{2C_f(i+1)}{(i+2)^2}+\delta
\le\tfrac{2C_f}{i+3}+\delta.
\end{align*}
where the last inequality follows from $(i+1)(i+3)\le(i+2)^2$. Hence, by induction,
\[
f(\Sigma^{(M)})-f(\Sigma^\star)=h_M\le\tfrac{2C_f}{M+2}+\delta,
\]
which proves \eqref{eq:obj-rate}.
Finally, strong convexity and first-order optimality of $\Sigma^\star$ give, for every $\Sigma\in\mathcal{D}$,
\[
f(\Sigma)\ge f(\Sigma^\star)
+\underbrace{\langle\nabla f(\Sigma^\star),
\Sigma-\Sigma^\star\rangle}_{\ge0}
+\tfrac{m}{2}\|\Sigma-\Sigma^\star\|_F^2.
\]
Applying this at $\Sigma=\Sigma^{(M)}$ and using \eqref{eq:obj-rate} yields \eqref{eq:iter-rate}.
\hfill$\blacksquare$
\subsection{Proof of Lemma~\ref{lem:subproblem}}\label{app:subproblem}
\emph{(i)} First, every $\lambda>\|M\|_2/\mu$ is admissible, which shows $\lambda_{\mathrm{crit}}\le\|M\|_2/\mu$. Indeed, for such $\lambda$ the stage weight satisfies $\widetilde M(\lambda)\succeq(\lambda\mu-\|M\|_2)I\succ0$. The cost is then nonnegative, so backward induction gives $P_t(\lambda)\succeq0$ for all $t$, and consequently $S_t(\lambda)=\widetilde M^{uu}(\lambda)+B^\top P_{t+1}(\lambda)B\succ0$.

Second, admissibility is preserved as $\lambda$ increases. When $S_t(\lambda)\succ0$, the Riccati step \eqref{eq:riccati} is the partial minimization
\[
x^\top P_t(\lambda)x=\min_{u}\{z^\top\widetilde M(\lambda)z+(Ax+Bu)^\top P_{t+1}(\lambda)(Ax+Bu)\},
\]
and the objective on the right is pointwise nondecreasing in $\lambda$, since its derivative in $\lambda$ is $z^\top W^{(2)}z\ge0$, and in $P_{t+1}\succeq 0$. Let $\lambda_1\in\Lambda$ and $\lambda_2>\lambda_1$, and induct backward from $P_H(\lambda_1)=P_H(\lambda_2)=0$. If $P_{t+1}(\lambda_2)\succeq P_{t+1}(\lambda_1)$, then $S_t(\lambda_2)\succeq S_t(\lambda_1)+(\lambda_2-\lambda_1)W^{(2),uu}\succ0$, so the minimization at $\lambda_2$ is well posed, and minimizing the pointwise larger objective yields $P_t(\lambda_2)\succeq P_t(\lambda_1)$. Hence $\lambda_2\in\Lambda$.
Finally, wherever all $S_t(\lambda)\succ0$ the recursion depends continuously on $\lambda$, so $\Lambda$ is open and therefore of the form $(\lambda_{\mathrm{crit}},\infty)$. Since $S_{H-1}(\lambda)=M^{uu}+\lambda W^{(2),uu}\prec0$ for all sufficiently small $\lambda\ge0$, it follows that $\lambda_{\mathrm{crit}}>0$.

\emph{(ii)} For $\lambda\in\Lambda$, all $S_t(\lambda)\succ0$, so the Riccati recursion \eqref{eq:riccati} involves only sums, products, and inverses of matrices that depend continuously on $\lambda$. Hence the gains $K_t(\lambda)$, and therefore $b(\lambda)$ through the covariance propagation following \eqref{eq:budget-map}, are continuous in $\lambda$. For monotonicity, fix $\lambda_1,\lambda_2\in\Lambda$ and define $c_j\coloneqq\Tr\bigl(M\cov{\pi_{\lambda_j}}\bigr)$ and $b_j\coloneqq b(\lambda_j)$. The cost of $\pi_{\lambda_j}$ under $\widetilde M(\lambda_i)$ is $c_j+\lambda_i b_j$. Since $\pi_{\lambda_1}$ and $\pi_{\lambda_2}$ are optimal for $\lambda_1$ and $\lambda_2$, respectively,
\[
c_1+\lambda_1b_1\le c_2+\lambda_1b_2,
\qquad
c_2+\lambda_2b_2\le c_1+\lambda_2b_1.
\]
Adding the two inequalities and cancelling $c_1+c_2$ yields
$
\lambda_1b_1+\lambda_2b_2
\le\lambda_1b_2+\lambda_2b_1,
$
which is equivalent to $(\lambda_2-\lambda_1)(b_2-b_1)\le0$.
Thus, $\lambda_2>\lambda_1$ implies $b(\lambda_2)\le b(\lambda_1)$.

\emph{(iii)} For $\lambda\in\Lambda$, $\pi_\lambda$ attains
$v(\lambda)=\inf_{\pi\in\Pi}\Tr\bigl(\widetilde M(\lambda)\cov{\pi}\bigr)$. For any feasible $\pi$ in \eqref{eq:lmo-control},
\[
\Tr\bigl(M\cov{\pi}\bigr)
=\Tr\bigl(\widetilde M(\lambda)\cov{\pi}\bigr)
-\lambda\Tr\bigl(W^{(2)}\cov{\pi}\bigr)
\ge v(\lambda)-\lambda\beta.
\]
The inequality follows from the definition of $v(\lambda)$ and from feasibility, since $\lambda\ge0$ implies
$-\lambda\Tr(W^{(2)}\cov{\pi})\ge-\lambda\beta$.
Hence $p^\star\ge v(\lambda)-\lambda\beta$. Evaluating at $\pi_\lambda$ gives
\begin{multline*}
\Tr\bigl(M\cov{\pi_\lambda}\bigr)
=v(\lambda)-\lambda b(\lambda)
=\bigl(v(\lambda)-\lambda\beta\bigr)
+\lambda\bigl(\beta-b(\lambda)\bigr)\\
\le p^\star+\lambda\bigl(\beta-b(\lambda)\bigr).
\end{multline*}
If $b(\lambda)\le\beta$, then $\pi_\lambda$ is feasible.
If $b(\lambda)=\beta$, the displayed bound gives $\Tr(M\cov{\pi_\lambda})\le p^\star$. Feasibility gives the reverse inequality by the definition of $p^\star$; hence equality holds, and $\pi_\lambda$ solves \eqref{eq:lmo-control}.

\emph{(iv)} We first upper bound $v(\lambda)$. Let $\pi_e\in\argmin_{\pi\in\Pi}\Tr\bigl(W^{(2)}\cov{\pi}\bigr)$, so that $\Tr\bigl(W^{(2)}\cov{\pi_e}\bigr)=b_0$. Since $M\prec0$ and $\cov{\pi_e}\succeq0$, we have $\Tr\bigl(M\cov{\pi_e}\bigr)\le0$. Using $\pi_e$ as a candidate policy in the definition of $v(\lambda)$,
\[
v(\lambda)\le\Tr\bigl(M\cov{\pi_e}\bigr)+\lambda b_0\le\lambda b_0 .
\]
We next lower bound $v(\lambda)$ in terms of $b(\lambda)$. Any $\Sigma\succeq0$ satisfies $\Tr(M\Sigma)\ge-\|M\|_2\Tr(\Sigma)$ and $\Tr(\Sigma)\le\Tr\bigl(W^{(2)}\Sigma\bigr)/\mu$. Applying both bounds to $\cov{\pi_\lambda}$
\[
v(\lambda)=\Tr\bigl(M\cov{\pi_\lambda}\bigr)+\lambda b(\lambda)\ge\Bigl(\lambda-\tfrac{\|M\|_2}{\mu}\Bigr)b(\lambda).
\]
Now take $\lambda>\|M\|_2/\mu$, which is admissible by \textit{(i)}. Chaining the two bounds on $v(\lambda)$ and dividing by $\lambda-\|M\|_2/\mu>0$ yields
$
b(\lambda)\le\tfrac{\lambda b_0}{\lambda-\|M\|_2/\mu}.
$
The right hand side is at most $\beta$ if and only if $\lambda b_0\le\beta\lambda-\beta\|M\|_2/\mu$, which rearranges to $\lambda(\beta-b_0)\ge\beta\|M\|_2/\mu$. Since $\beta>b_0$ by Assumption~\ref{ass:standing}(ii), this holds exactly when $\lambda\ge\bar\lambda$. \hfill$\blacksquare$

\bibliographystyle{IEEEtran}
\bibliography{refs}

\end{document}

%% file: fw_objective_vs_beta.tex
\begin{tikzpicture}

\definecolor{crimson2143940}{RGB}{214,39,40}
\definecolor{darkgrey176}{RGB}{176,176,176}
\definecolor{darkorange25512714}{RGB}{255,127,14}
\definecolor{forestgreen4416044}{RGB}{44,160,44}
\definecolor{lightgrey204}{RGB}{204,204,204}
\definecolor{mediumpurple148103189}{RGB}{148,103,189}
\definecolor{sienna1408675}{RGB}{140,86,75}
\definecolor{steelblue31119180}{RGB}{31,119,180}

\begin{axis}[
	legend style={
		at={(0.5,1.08)},
		anchor=south,
		legend columns=3,
		draw=black,
		fill=white,
		font=\small
	},
	legend cell align={left},
	yticklabel style={text width=2.2em, align=right},
	log basis y={10},
	tick align=outside,
	tick pos=left,
	x grid style={darkgrey176},
	xlabel={Frank--Wolfe iteration},
	xmajorgrids,
	xmin=-2.5, xmax=52.5,
	xminorgrids,
	xtick style={color=black},
	y grid style={darkgrey176},
ylabel={Objective $\mathrm{Tr}(W^{(1)} \Sigma^{-1})$},
	ymajorgrids,
	ymin=0.00578731937763324, ymax=18.3434700105093,
	yminorgrids,
	ymode=log,
	height=4.5cm,
	width=8cm,
	ytick style={color=black},
	ytick={0.0001,0.001,0.01,0.1,1,10,100},
	yticklabels={
		$10^{-4}$,
		$10^{-3}$,
		$10^{-2}$,
		$10^{-1}$,
		$10^{0}$,
		$10^{1}$,
		$10^{2}$
	}
	]
\addplot [semithick, steelblue31119180, mark=*, mark size=2, mark options={solid}]
table {%
0 12.7158368350352
1 0.490714699539539
2 0.463458136175405
3 0.453354564548036
4 0.446621821387625
5 0.443353570836086
6 0.44158565321837
7 0.440521517244158
8 0.439828675618959
9 0.4393512348745
10 0.43900777540655
11 0.438752210325832
12 0.438556782469142
13 0.4384039245542
14 0.438282076704158
15 0.438183356325971
16 0.438102242805799
17 0.43803477816211
18 0.437978054605776
19 0.437929905855469
20 0.437888681811128
21 0.437853114452402
22 0.437822209949337
23 0.43779518960644
24 0.437771427090322
25 0.437750417452054
26 0.437731751090394
27 0.437715091528595
28 0.437700163083723
29 0.437686730772976
30 0.437674601819407
31 0.437663614915042
32 0.437653628844879
33 0.437644525155805
34 0.437636205267517
35 0.437628580169803
36 0.437621574509075
37 0.437615122816121
38 0.437609166932867
39 0.437603658765936
40 0.437598555508259
41 0.437593816870349
42 0.437589409054575
43 0.437585302391527
44 0.437581469143359
45 0.437577887334744
46 0.437574534386732
47 0.437571390920148
48 0.437568440441082
49 0.437565666443589
50 0.437563056829641
};
\addlegendentry{$\beta=20.0$}
\addplot [semithick, darkorange25512714, mark=*, mark size=2, mark options={solid}]
table {%
0 12.7158368350352
1 0.341674364307712
2 0.195545221766927
3 0.192528314721017
4 0.162324208474271
5 0.170045960467213
6 0.167831859571785
7 0.161071535288444
8 0.160084089419884
9 0.161115751925631
10 0.158265662119286
11 0.160060844678334
12 0.158256545429488
13 0.163306108980121
14 0.16058868156988
15 0.159174209746517
16 0.15826201190521
17 0.158062107083491
18 0.158011497906918
19 0.157861745230822
20 0.159221603245701
21 0.158239973665126
22 0.158061725270349
23 0.15770541121603
24 0.15784595904884
25 0.158013978810897
26 0.157646528630778
27 0.157710880324825
28 0.157651278291323
29 0.157601724645514
30 0.157753674173621
31 0.157537344343555
32 0.157702994403211
33 0.157544596084343
34 0.158365956554618
35 0.157975310460331
36 0.157754782342362
37 0.157581179192297
38 0.157560986576894
39 0.157540152697394
40 0.157529541778883
41 0.157770448467036
42 0.157561710512833
43 0.157561798270026
44 0.15748602976401
45 0.15753768517166
46 0.157657798021079
47 0.15752597007836
48 0.157521539368114
49 0.157465867402326
50 0.157499670028985
};
\addlegendentry{$\beta=50.0$}
\addplot [semithick, forestgreen4416044, mark=*, mark size=2, mark options={solid}]
table {%
0 12.7158368350352
1 0.314123126710525
2 0.15304867356937
3 0.0926150026699704
4 0.0875947984205825
5 0.0850672175217101
6 0.0836396963070882
7 0.0835818309696105
8 0.0892422215184939
9 0.0851581261331926
10 0.0827347493238472
11 0.0812507787116488
12 0.0809042980897075
13 0.0805409694343543
14 0.0810343592547921
15 0.0806103574770895
16 0.080718044071103
17 0.080562702987848
18 0.0819444374275772
19 0.0813098915777399
20 0.0805357221455517
21 0.080397741216808
22 0.0801346445408081
23 0.0802994810319118
24 0.0800831553523451
25 0.0812298809294623
26 0.080949699191002
27 0.0803841874301607
28 0.0802028803219003
29 0.0800731817444963
30 0.08012702695088
31 0.0800484654888634
32 0.0806703549474173
33 0.0804778532066339
34 0.0801719847400931
35 0.0800893352439207
36 0.079999554110224
37 0.0800561704479776
38 0.0799873353168143
39 0.0804481492142901
40 0.080318175245416
41 0.0800990296678014
42 0.0800323002653218
43 0.079970726454646
44 0.0800060957310955
45 0.0799618420790463
46 0.0802945259285089
47 0.0801961464901629
48 0.0800415364995159
49 0.0799950423130217
50 0.0799474697945485
};
\addlegendentry{$\beta=100.0$}
\addplot [semithick, crimson2143940, mark=*, mark size=2, mark options={solid}]
table {%
0 12.7158368350352
1 0.28057156668071
2 0.144257356393817
3 0.0662925369144625
4 0.0459479808686282
5 0.0431914691209027
6 0.0418340002439786
7 0.0412769582143555
8 0.0423795789629226
9 0.042111076515806
10 0.0413077955312062
11 0.0412607273577652
12 0.0409044313274533
13 0.0408174070797419
14 0.0407796947847059
15 0.0407097433839374
16 0.0419892342416723
17 0.041633754417988
18 0.0409723410348038
19 0.0407878305085597
20 0.040631709966091
21 0.0406830475312225
22 0.0406031577593872
23 0.0413224180796697
24 0.0411336019952365
25 0.040786303046841
26 0.0406771039743628
27 0.0405903536372011
28 0.0406186654578895
29 0.0405730719590205
30 0.0410136654669604
31 0.0408922942795507
32 0.0406888054329903
33 0.0406218749056128
34 0.0405658714981791
35 0.0405858019009373
36 0.0405548847812
37 0.0408582220733974
38 0.040774095411079
39 0.0406385608617523
40 0.0405918705061531
41 0.0405523755853261
42 0.0405658980239859
43 0.0405440520051881
44 0.0407658916608007
45 0.0407036373136627
46 0.0406072051389759
47 0.040572777192694
48 0.0405432489190845
49 0.0405528725077219
50 0.0405366053721917
};
\addlegendentry{$\beta=200.0$}
\addplot [semithick, mediumpurple148103189, mark=*, mark size=2, mark options={solid}]
table {%
0 12.7158368350352
1 0.226458418856089
2 0.0690758543317906
3 0.0297523397996977
4 0.0233180957255151
5 0.0177399370601745
6 0.017269759921577
7 0.0166383378531527
8 0.0168056228392551
9 0.0165081686523555
10 0.0169213474855995
11 0.0167781588883183
12 0.0166881413680013
13 0.0166293837539233
14 0.0165665744752246
15 0.0165036654032048
16 0.0164684689413247
17 0.016455426436781
18 0.0169368014061454
19 0.0167869513002598
20 0.0165714153930281
21 0.0164976451996484
22 0.016437251460454
23 0.0164425984330168
24 0.0164206286098386
25 0.0167019518985067
26 0.0166221044588785
27 0.0165006566157247
28 0.0164567329695683
29 0.0164214911486013
30 0.0164241040657559
31 0.0164109762316879
32 0.0165929857020076
33 0.0165403241171208
34 0.0164655716859558
35 0.01643638931621
36 0.016412981206725
37 0.0164140100320402
38 0.016405189161824
39 0.0165336368651816
40 0.0164959932092322
41 0.0164453795517378
42 0.0164244674301793
43 0.0164077215696248
44 0.0164081307793414
45 0.0164017858994212
46 0.0164981364908097
47 0.0164700797659769
48 0.016433530807139
49 0.0164179457498782
50 0.0164053195677535
};
\addlegendentry{$\beta=500.0$}
\addplot [semithick, sienna1408675, mark=*, mark size=2, mark options={solid}]
table {%
0 12.7158368350352
1 0.18941001664155
2 0.0449676327668598
3 0.0153475445609215
4 0.0118219116385636
5 0.00894642495135621
6 0.00869221970242919
7 0.00836542189403235
8 0.00844308962025676
9 0.00828941278488067
10 0.00855924551002416
11 0.00849278992095421
12 0.00838493542138241
13 0.00835038297284177
14 0.00830278229258562
15 0.00828206553516444
16 0.00857191157795366
17 0.0084702773822841
18 0.00833882162937031
19 0.00829601651269465
20 0.00825843839105963
21 0.00826456902699122
22 0.00825041855365574
23 0.00841805518035608
24 0.00836990310704927
25 0.00829745757640293
26 0.00827223119135556
27 0.00825133668088227
28 0.00825253941002375
29 0.00824514555900214
30 0.00835730726557398
31 0.00832584771408613
32 0.00828335517586527
33 0.00826676525622776
34 0.00825318342182892
35 0.00825356164642405
36 0.00824862606219355
37 0.00833088597634513
38 0.00830943983578795
39 0.00828116943833169
40 0.00826974968463009
41 0.00826017318415955
42 0.00826036328682223
43 0.00825628304792103
44 0.0083217270437722
45 0.00830587889238456
46 0.00828535826668587
47 0.0082765006143188
48 0.00826892939447576
49 0.00826832590212329
50 0.00826474284092094
};
\addlegendentry{$\beta=1000.0$}
\end{axis}

\end{tikzpicture}

%% file: system_id_error_mc.tex
\begin{tikzpicture}

\definecolor{crimson2143940}{RGB}{214,39,40}
\definecolor{darkgrey176}{RGB}{176,176,176}
\definecolor{darkorange25512714}{RGB}{255,127,14}
\definecolor{forestgreen4416044}{RGB}{44,160,44}
\definecolor{lightgrey204}{RGB}{204,204,204}
\definecolor{steelblue31119180}{RGB}{31,119,180}

\begin{axis}[
	legend style={
		at={(0.5,1.08)},
		anchor=south,
		legend columns=2,
		draw=black,
		fill=white,
		font=\small
	},
	legend cell align={left},
	yticklabel style={text width=2.2em, align=right},
	log basis y={10},
	tick align=outside,
	tick pos=left,
	x grid style={darkgrey176},
	xlabel={Episode},
	xmajorgrids,
	xmin=-1.4, xmax=51.4,
	xminorgrids,
	xtick style={color=black},
	y grid style={darkgrey176},
	ylabel={Error $\|[\hat A\ \hat B]-[A\ B]\|_F$},
	ymajorgrids,
	ymin=6.09885643429937e-06, ymax=24075.1627427264,
	yminorgrids,
	ymode=log,
	height=4.5cm,
	width=8cm,
	ytick style={color=black},
	ytick={1e-08,1e-06,0.0001,0.01,1,100,10000,1000000,100000000},
	yticklabels={
		$10^{-8}$,
		$10^{-6}$,
		$10^{-4}$,
		$10^{-2}$,
		$10^{0}$,
		$10^{2}$,
		$10^{4}$,
		$10^{6}$,
		$10^{8}$
	}
	]
\path [fill=steelblue31119180, fill opacity=0.22]
(axis cs:1,22556.4705014433)
--(axis cs:1,407.183720800482)
--(axis cs:2,0.115220748311312)
--(axis cs:3,0.000176385593996459)
--(axis cs:4,9.62000372177126e-05)
--(axis cs:5,1.44542583176946e-05)
--(axis cs:6,1.44539275117826e-05)
--(axis cs:7,1.44539951111114e-05)
--(axis cs:8,1.44544633019999e-05)
--(axis cs:9,1.44536467698097e-05)
--(axis cs:10,1.44541805651078e-05)
--(axis cs:11,1.44543759910938e-05)
--(axis cs:12,1.44540472725737e-05)
--(axis cs:13,1.44539496636045e-05)
--(axis cs:14,1.44529158251801e-05)
--(axis cs:15,1.44532136927243e-05)
--(axis cs:16,1.44532656410125e-05)
--(axis cs:17,1.44525674821723e-05)
--(axis cs:18,1.44526422000752e-05)
--(axis cs:19,1.44528271125633e-05)
--(axis cs:20,1.44534762849744e-05)
--(axis cs:21,1.44530049754854e-05)
--(axis cs:22,1.44528354351104e-05)
--(axis cs:23,1.44527065516006e-05)
--(axis cs:24,1.44533523615962e-05)
--(axis cs:25,1.44527818178065e-05)
--(axis cs:26,1.44521211003552e-05)
--(axis cs:27,1.44531609129398e-05)
--(axis cs:28,1.44532668265558e-05)
--(axis cs:29,1.44532533847805e-05)
--(axis cs:30,1.44533986094511e-05)
--(axis cs:31,1.44533279167193e-05)
--(axis cs:32,1.44537838061985e-05)
--(axis cs:33,1.44533906860786e-05)
--(axis cs:34,1.44541700464051e-05)
--(axis cs:35,1.44545251364541e-05)
--(axis cs:36,1.44547635551146e-05)
--(axis cs:37,1.44546174434352e-05)
--(axis cs:38,1.44556712993788e-05)
--(axis cs:39,1.4455234715654e-05)
--(axis cs:40,1.44557124308344e-05)
--(axis cs:41,1.44559769145183e-05)
--(axis cs:42,1.44565197351775e-05)
--(axis cs:43,1.4457097729109e-05)
--(axis cs:44,1.44567529475931e-05)
--(axis cs:45,1.44568805660111e-05)
--(axis cs:46,1.44573114370408e-05)
--(axis cs:47,1.44580594613111e-05)
--(axis cs:48,1.44574924454737e-05)
--(axis cs:49,1.44583110158075e-05)
--(axis cs:49,0.0131002112426105)
--(axis cs:49,0.0131002112426105)
--(axis cs:48,0.0133967832824658)
--(axis cs:47,0.0134019582509735)
--(axis cs:46,0.013572997112819)
--(axis cs:45,0.0139937384230433)
--(axis cs:44,0.013580822784668)
--(axis cs:43,0.0145050699207943)
--(axis cs:42,0.014223278175052)
--(axis cs:41,0.0139187549090937)
--(axis cs:40,0.012368248654509)
--(axis cs:39,0.012211382258274)
--(axis cs:38,0.0119096654781589)
--(axis cs:37,0.0116825967814306)
--(axis cs:36,0.0116985116025066)
--(axis cs:35,0.0112885032371303)
--(axis cs:34,0.0100682133455741)
--(axis cs:33,0.0105449094539479)
--(axis cs:32,0.00995659031811535)
--(axis cs:31,0.0130304982692401)
--(axis cs:30,0.0129897071484585)
--(axis cs:29,0.01284541015187)
--(axis cs:28,0.0128234703128984)
--(axis cs:27,0.0176878322000747)
--(axis cs:26,0.0173936590290431)
--(axis cs:25,0.0182221071968933)
--(axis cs:24,0.0186985626635954)
--(axis cs:23,0.020409518290214)
--(axis cs:22,0.01974323831238)
--(axis cs:21,0.0195462252554497)
--(axis cs:20,0.0167599814132092)
--(axis cs:19,0.0220935244255244)
--(axis cs:18,0.023031992155652)
--(axis cs:17,0.0196086984828925)
--(axis cs:16,0.0227575599755086)
--(axis cs:15,0.0279058407401953)
--(axis cs:14,0.0274965446995349)
--(axis cs:13,0.0299093417247008)
--(axis cs:12,0.0321947786847844)
--(axis cs:11,0.0361033191518233)
--(axis cs:10,0.048822034823175)
--(axis cs:9,0.055332805760498)
--(axis cs:8,0.0986831642303836)
--(axis cs:7,0.246148347501204)
--(axis cs:6,0.23490070252047)
--(axis cs:5,0.331671094825258)
--(axis cs:4,0.409650773077233)
--(axis cs:3,43.061470417886)
--(axis cs:2,333.522784816945)
--(axis cs:1,22556.4705014433)
--cycle;

\path [fill=darkorange25512714, fill opacity=0.22]
(axis cs:1,22556.4705014433)
--(axis cs:1,407.183720800482)
--(axis cs:2,34.2630806631708)
--(axis cs:3,1.35726682794714)
--(axis cs:4,0.0165664321609906)
--(axis cs:5,0.000592446632132264)
--(axis cs:6,0.000592203306467158)
--(axis cs:7,0.000592682066521383)
--(axis cs:8,0.000592917077608227)
--(axis cs:9,0.000593591328619602)
--(axis cs:10,0.000593858801904272)
--(axis cs:11,0.000594573042779465)
--(axis cs:12,0.000595376637281068)
--(axis cs:13,0.000595726686657545)
--(axis cs:14,0.000595346550758328)
--(axis cs:15,0.00059486848953296)
--(axis cs:16,0.000594516340583921)
--(axis cs:17,0.000594605518136663)
--(axis cs:18,0.000595019717705189)
--(axis cs:19,0.000594988992813435)
--(axis cs:20,0.000594970005680505)
--(axis cs:21,0.000594750796780858)
--(axis cs:22,0.000594599414510519)
--(axis cs:23,0.000595798019628315)
--(axis cs:24,0.000595782868614508)
--(axis cs:25,0.000595490084200659)
--(axis cs:26,0.000595377005652489)
--(axis cs:27,0.000594888127447096)
--(axis cs:28,0.000594542699726132)
--(axis cs:29,0.00059498597582295)
--(axis cs:30,0.000593989823341501)
--(axis cs:31,0.000594199578925321)
--(axis cs:32,0.000594169784149037)
--(axis cs:33,0.000593776747454885)
--(axis cs:34,0.000594137970735953)
--(axis cs:35,0.000594067440406416)
--(axis cs:36,0.00059484291903993)
--(axis cs:37,0.000594605500369366)
--(axis cs:38,0.000594888238084758)
--(axis cs:39,0.000594870326510039)
--(axis cs:40,0.000594902308128517)
--(axis cs:41,0.000595048164519782)
--(axis cs:42,0.000595080152900752)
--(axis cs:43,0.000594932979767767)
--(axis cs:44,0.000594788481930764)
--(axis cs:45,0.000595686295510761)
--(axis cs:46,0.000595661276980933)
--(axis cs:47,0.000595078381204571)
--(axis cs:48,0.000595436253635744)
--(axis cs:49,0.00059560673106352)
--(axis cs:49,0.219627397562889)
--(axis cs:49,0.219627397562889)
--(axis cs:48,0.216735785637893)
--(axis cs:47,0.208383527883006)
--(axis cs:46,0.21803618011961)
--(axis cs:45,0.220344703388591)
--(axis cs:44,0.225923099313119)
--(axis cs:43,0.241113363972789)
--(axis cs:42,0.247649920917315)
--(axis cs:41,0.255929378444236)
--(axis cs:40,0.263321026863391)
--(axis cs:39,0.260934477501778)
--(axis cs:38,0.244338508222805)
--(axis cs:37,0.252349666704092)
--(axis cs:36,0.250981956026984)
--(axis cs:35,0.254362421035475)
--(axis cs:34,0.254747405282846)
--(axis cs:33,0.28348535729557)
--(axis cs:32,0.264775005483491)
--(axis cs:31,0.270945862978866)
--(axis cs:30,0.268759636255436)
--(axis cs:29,0.261827322643912)
--(axis cs:28,0.277122735965296)
--(axis cs:27,0.268884354899952)
--(axis cs:26,0.262442894382814)
--(axis cs:25,0.270725037852359)
--(axis cs:24,0.259004181143591)
--(axis cs:23,0.266708493054659)
--(axis cs:22,0.287160654287914)
--(axis cs:21,0.270169449072553)
--(axis cs:20,0.268885965202372)
--(axis cs:19,0.308364185709337)
--(axis cs:18,0.294144556106007)
--(axis cs:17,0.276560467685606)
--(axis cs:16,0.304482253737488)
--(axis cs:15,0.373337197316497)
--(axis cs:14,0.353263052604707)
--(axis cs:13,0.32467446756472)
--(axis cs:12,0.333546002273005)
--(axis cs:11,0.352377335849737)
--(axis cs:10,0.351350235604434)
--(axis cs:9,0.415611192217914)
--(axis cs:8,0.450545758071496)
--(axis cs:7,0.462196965043269)
--(axis cs:6,0.432506158052721)
--(axis cs:5,1.07032317182482)
--(axis cs:4,38.6556855290496)
--(axis cs:3,381.522370244303)
--(axis cs:2,5563.7878561775)
--(axis cs:1,22556.4705014433)
--cycle;

\path [fill=forestgreen4416044, fill opacity=0.22]
(axis cs:1,22556.4705014433)
--(axis cs:1,407.183720800482)
--(axis cs:2,12.0629979457021)
--(axis cs:3,1.26102592670867)
--(axis cs:4,0.218527605372697)
--(axis cs:5,0.0781425433155865)
--(axis cs:6,0.089468708965216)
--(axis cs:7,0.0716831172201107)
--(axis cs:8,0.0524423322847326)
--(axis cs:9,0.0458996477246305)
--(axis cs:10,0.0191735527155294)
--(axis cs:11,0.0174744929173153)
--(axis cs:12,0.0221911055132578)
--(axis cs:13,0.0322831943544374)
--(axis cs:14,0.0405391313206468)
--(axis cs:15,0.0319829227568918)
--(axis cs:16,0.027375508954007)
--(axis cs:17,0.0245786615627951)
--(axis cs:18,0.0202489949552552)
--(axis cs:19,0.0263563723571451)
--(axis cs:20,0.0263449275695527)
--(axis cs:21,0.0132030259485203)
--(axis cs:22,0.0217220427557147)
--(axis cs:23,0.0330375921787195)
--(axis cs:24,0.02727542590534)
--(axis cs:25,0.0238251576479238)
--(axis cs:26,0.0278990141763631)
--(axis cs:27,0.0199572424693074)
--(axis cs:28,0.0275334554814292)
--(axis cs:29,0.0280520141821473)
--(axis cs:30,0.0185279015284205)
--(axis cs:31,0.0194610984450682)
--(axis cs:32,0.0162524596256366)
--(axis cs:33,0.0127588761896759)
--(axis cs:34,0.0179890574397234)
--(axis cs:35,0.0211699301506032)
--(axis cs:36,0.0182130749665644)
--(axis cs:37,0.025689075314598)
--(axis cs:38,0.0243442840777025)
--(axis cs:39,0.0227724938040263)
--(axis cs:40,0.0213680915696914)
--(axis cs:41,0.0182134532350441)
--(axis cs:42,0.0149960797948583)
--(axis cs:43,0.0129571546163445)
--(axis cs:44,0.0064465437686357)
--(axis cs:45,0.00545784226880658)
--(axis cs:46,0.0101074351183937)
--(axis cs:47,0.00913755343946641)
--(axis cs:48,0.00810339525392629)
--(axis cs:49,0.0107930056092254)
--(axis cs:49,0.152786280855003)
--(axis cs:49,0.152786280855003)
--(axis cs:48,0.154033562619983)
--(axis cs:47,0.151203674671413)
--(axis cs:46,0.146046017682959)
--(axis cs:45,0.13938317082686)
--(axis cs:44,0.143552465224724)
--(axis cs:43,0.146234072032793)
--(axis cs:42,0.145365483340894)
--(axis cs:41,0.154260817775359)
--(axis cs:40,0.146332025024531)
--(axis cs:39,0.14246955496666)
--(axis cs:38,0.136594273506606)
--(axis cs:37,0.136695408384166)
--(axis cs:36,0.135157626464279)
--(axis cs:35,0.152913352896439)
--(axis cs:34,0.170413517279249)
--(axis cs:33,0.165018895352243)
--(axis cs:32,0.170283001344083)
--(axis cs:31,0.173939022813556)
--(axis cs:30,0.169414339502497)
--(axis cs:29,0.178312326348694)
--(axis cs:28,0.168161729118497)
--(axis cs:27,0.154343797534332)
--(axis cs:26,0.171247000203792)
--(axis cs:25,0.194393762864453)
--(axis cs:24,0.19047571350833)
--(axis cs:23,0.184265734525783)
--(axis cs:22,0.175722625808076)
--(axis cs:21,0.162305734992599)
--(axis cs:20,0.169611061651742)
--(axis cs:19,0.167351438555673)
--(axis cs:18,0.189935267097294)
--(axis cs:17,0.225194823281541)
--(axis cs:16,0.26668289213485)
--(axis cs:15,0.260585527439692)
--(axis cs:14,0.218210430508212)
--(axis cs:13,0.268540219641699)
--(axis cs:12,0.339128573715467)
--(axis cs:11,0.359833435576198)
--(axis cs:10,0.417392396841173)
--(axis cs:9,0.815446899983506)
--(axis cs:8,1.31566041359482)
--(axis cs:7,3.82908084131544)
--(axis cs:6,51.0286106234076)
--(axis cs:5,155.633524442213)
--(axis cs:4,906.140431626747)
--(axis cs:3,1722.01212800394)
--(axis cs:2,8398.19488952522)
--(axis cs:1,22556.4705014433)
--cycle;

\path [fill=crimson2143940, fill opacity=0.22]
(axis cs:1,22556.4705014433)
--(axis cs:1,407.183720800482)
--(axis cs:2,9.6203027762131)
--(axis cs:3,0.234365863970734)
--(axis cs:4,0.0628175258522219)
--(axis cs:5,0.0564412588212595)
--(axis cs:6,0.0472428425897344)
--(axis cs:7,0.0424267900894973)
--(axis cs:8,0.0399052377281693)
--(axis cs:9,0.0177757821268643)
--(axis cs:10,0.0370865954793797)
--(axis cs:11,0.0299919976563303)
--(axis cs:12,0.0255855710865028)
--(axis cs:13,0.0203275882614264)
--(axis cs:14,0.0140141591745965)
--(axis cs:15,0.00788654757402403)
--(axis cs:16,0.016826192186336)
--(axis cs:17,0.0185341710929999)
--(axis cs:18,0.0192382972210841)
--(axis cs:19,0.016967778371246)
--(axis cs:20,0.0111632556997884)
--(axis cs:21,0.020448323189983)
--(axis cs:22,0.0177404273533847)
--(axis cs:23,0.0181544918868037)
--(axis cs:24,0.0245722494040585)
--(axis cs:25,0.0222245938636368)
--(axis cs:26,0.0189596977044469)
--(axis cs:27,0.0178993390893906)
--(axis cs:28,0.0123841287391404)
--(axis cs:29,0.0159349853516892)
--(axis cs:30,0.0111749800146599)
--(axis cs:31,0.00919652793911135)
--(axis cs:32,0.013594925380265)
--(axis cs:33,0.015011111725715)
--(axis cs:34,0.0119000201398242)
--(axis cs:35,0.0121030728341571)
--(axis cs:36,0.0101756319574685)
--(axis cs:37,0.00765037105728042)
--(axis cs:38,0.0106838839650206)
--(axis cs:39,0.0102254059580001)
--(axis cs:40,0.0099065816132933)
--(axis cs:41,0.0102521929002533)
--(axis cs:42,0.0104755511834668)
--(axis cs:43,0.00745738211426573)
--(axis cs:44,0.0107581998642706)
--(axis cs:45,0.00952738336614373)
--(axis cs:46,0.00712443417798505)
--(axis cs:47,0.00631485142576126)
--(axis cs:48,0.00788846337062857)
--(axis cs:49,0.00852634213062618)
--(axis cs:49,0.0614528855997983)
--(axis cs:49,0.0614528855997983)
--(axis cs:48,0.0699725199297147)
--(axis cs:47,0.0681332737768378)
--(axis cs:46,0.0714531849704481)
--(axis cs:45,0.0739157832377567)
--(axis cs:44,0.0803042591440448)
--(axis cs:43,0.0805742855367958)
--(axis cs:42,0.0838147067728047)
--(axis cs:41,0.0862844307513022)
--(axis cs:40,0.0802304129752551)
--(axis cs:39,0.0868517599131418)
--(axis cs:38,0.0847623717058588)
--(axis cs:37,0.0829628489661172)
--(axis cs:36,0.0868354335736454)
--(axis cs:35,0.0778457763218306)
--(axis cs:34,0.0773758072355336)
--(axis cs:33,0.0801884806014874)
--(axis cs:32,0.0839344308449178)
--(axis cs:31,0.0878317708958321)
--(axis cs:30,0.0756887855180687)
--(axis cs:29,0.0754672009378663)
--(axis cs:28,0.0765534978623381)
--(axis cs:27,0.0826246222151288)
--(axis cs:26,0.080627007603795)
--(axis cs:25,0.0877139475106038)
--(axis cs:24,0.094459492652428)
--(axis cs:23,0.0916636149131164)
--(axis cs:22,0.096790409703012)
--(axis cs:21,0.0857305432789539)
--(axis cs:20,0.0970652590781142)
--(axis cs:19,0.105579099707369)
--(axis cs:18,0.116940095400781)
--(axis cs:17,0.118565667546504)
--(axis cs:16,0.125771381886615)
--(axis cs:15,0.149573306445826)
--(axis cs:14,0.149076821858759)
--(axis cs:13,0.157936539269829)
--(axis cs:12,0.139226734740834)
--(axis cs:11,0.147487184099233)
--(axis cs:10,0.163191812643162)
--(axis cs:9,0.177306030272336)
--(axis cs:8,0.16127065059564)
--(axis cs:7,0.215492412571458)
--(axis cs:6,0.370015239966938)
--(axis cs:5,1.56018079859639)
--(axis cs:4,22.0292385256266)
--(axis cs:3,211.221652500579)
--(axis cs:2,3212.94522252601)
--(axis cs:1,22556.4705014433)
--cycle;

\addplot [thick, steelblue31119180, mark=*, mark size=2, mark options={solid}]
table {%
1 4423.5677045818
2 0.802515219249607
3 0.090455714302991
4 0.0317217902888794
5 0.00922314548821395
6 0.00447246129819753
7 0.00401420769993951
8 0.00446871307159496
9 0.00355283827145312
10 0.00352238272032726
11 0.00354395757611231
12 0.00353897406864333
13 0.00353679847737502
14 0.0031476121368035
15 0.003106489804214
16 0.00311853192470506
17 0.00305452684811807
18 0.0030282168460771
19 0.00303069487452806
20 0.00312668024279021
21 0.00285753323503117
22 0.00283702260579649
23 0.00283851678113471
24 0.00292133266963392
25 0.00295968687995049
26 0.00297491685397304
27 0.00292061497961954
28 0.00301443368705065
29 0.00274256561916729
30 0.00301468634580451
31 0.00313644572022061
32 0.00313559009832094
33 0.00313332715910941
34 0.00313141958396849
35 0.0030680811303355
36 0.00275219112409226
37 0.00299453811489588
38 0.0031008472871239
39 0.00314179268084823
40 0.00308915271802873
41 0.00314581998398154
42 0.00314309533916239
43 0.00314526585518279
44 0.00314295952818747
45 0.00314278480855478
46 0.00309722277958994
47 0.00313082922919142
48 0.00307764023366969
49 0.002725163741167
};
\addlegendentry{Frank-Wolfe}
\addplot [thick, darkorange25512714, mark=*, mark size=2, mark options={solid}]
table {%
1 4423.5677045818
2 297.52346083078
3 19.9870667363263
4 0.767833078616538
5 0.192886248311881
6 0.131202946288508
7 0.129372386870281
8 0.138663741660444
9 0.101596786903366
10 0.0901582948443954
11 0.0927633851813182
12 0.0918523032966616
13 0.0880224797462466
14 0.0890945877024469
15 0.0869315739914978
16 0.0830676247269735
17 0.082398314290204
18 0.0793874584736308
19 0.0744169484756212
20 0.0742795924032663
21 0.0850123337526946
22 0.0894680714530348
23 0.0837163590497704
24 0.0836044813713544
25 0.0782317164807415
26 0.0798213332934788
27 0.0798465326212608
28 0.0759563453746821
29 0.0747404967483141
30 0.0710000412529947
31 0.0739915762455713
32 0.0737275467408978
33 0.0752815224084023
34 0.0743373153548336
35 0.0710430151014774
36 0.0733133859937323
37 0.0748795286963088
38 0.0748128398645296
39 0.0757019450150006
40 0.0770954840594173
41 0.0730678018217843
42 0.071560083152599
43 0.0673573159370638
44 0.0640493385606207
45 0.062904325726746
46 0.0639118085166376
47 0.0641134809633657
48 0.0635958201802944
49 0.0610490146638474
};
\addlegendentry{Certainty equivalence}
\addplot [thick, forestgreen4416044, mark=*, mark size=2, mark options={solid}]
table {%
1 4423.5677045818
2 618.684346707564
3 93.905124949469
4 13.6513806194273
5 3.12967140524065
6 0.562145584658519
7 0.25103169830238
8 0.174191632360288
9 0.150151464367199
10 0.140435638984922
11 0.106214670553337
12 0.100618487024824
13 0.0975504338001959
14 0.100416941427565
15 0.0918070414596397
16 0.0832832312839729
17 0.0769545912525383
18 0.0721214956668451
19 0.0636735063889708
20 0.0607306678911523
21 0.0670365634776866
22 0.0639504538955669
23 0.069856813485365
24 0.0685680758289616
25 0.0684639046499737
26 0.065995174952554
27 0.0595138002419472
28 0.0657086585971623
29 0.0609969118844904
30 0.0625461173148127
31 0.0551183834674314
32 0.0615468845649288
33 0.0554387211728838
34 0.0535725254773305
35 0.0485407514499697
36 0.0537030663504195
37 0.0542887020634846
38 0.0563328925366229
39 0.0553874660348117
40 0.0487458085973411
41 0.0520050067752625
42 0.0493135326281609
43 0.0464489272064491
44 0.0491275525242879
45 0.0526873530852557
46 0.0507182012023384
47 0.0465052661884147
48 0.0481657035041588
49 0.0480945553399098
};
\addlegendentry{Naive exploration}
\addplot [thick, crimson2143940, mark=*, mark size=2, mark options={solid}]
table {%
1 4423.5677045818
2 217.710211031696
3 10.7565929538293
4 0.642002504627826
5 0.143300175555298
6 0.116840110790837
7 0.0994518471167743
8 0.0856539665100032
9 0.0700707709898022
10 0.0656707391914755
11 0.0659088083028319
12 0.0686767147287812
13 0.0600967302480754
14 0.0576020561732986
15 0.0593643830468706
16 0.0538404607858313
17 0.0494801452180155
18 0.0458632401516566
19 0.0421858630292564
20 0.0437795393096511
21 0.0398996065337234
22 0.0378828464829734
23 0.038653298902849
24 0.0402250305679907
25 0.0362123483143411
26 0.0368057924275744
27 0.0349058939828343
28 0.0344454820782023
29 0.0355107058222609
30 0.0329065402273989
31 0.0330134669542929
32 0.0318602898102894
33 0.0329405715788392
34 0.0314796735992413
35 0.0325901026418791
36 0.0308531013657797
37 0.0325892739614378
38 0.0322555164927808
39 0.0322036371568365
40 0.0333781853917243
41 0.0315400437939324
42 0.0316300599140753
43 0.0296242457914295
44 0.0295795276265794
45 0.028414045279546
46 0.0281999041281452
47 0.0285738492156835
48 0.0288976705995582
49 0.0271229732659293
};
\addlegendentry{Frequency-based}
\end{axis}

\end{tikzpicture}

%% file: system_id_cost_mc.tex
\begin{tikzpicture}

\definecolor{crimson2143940}{RGB}{214,39,40}
\definecolor{darkgrey176}{RGB}{176,176,176}
\definecolor{darkorange25512714}{RGB}{255,127,14}
\definecolor{forestgreen4416044}{RGB}{44,160,44}
\definecolor{lightgrey204}{RGB}{204,204,204}
\definecolor{steelblue31119180}{RGB}{31,119,180}

\begin{axis}[
 legend style={
        at={(0.5,1.08)},
        anchor=south,
        legend columns=2,
        draw=black,
        fill=white,
        font=\small
    },
    legend cell align={left},
yticklabel style={text width=2.2em, align=right},
log basis y={10},
tick align=outside,
tick pos=left,
x grid style={darkgrey176},
xlabel={Episode},
xmajorgrids,
xmin=-1.4, xmax=51.4,
xminorgrids,
xtick style={color=black},
y grid style={darkgrey176},
height=4.5cm,
width=8cm,
ylabel={Objective $\operatorname{Tr}((\Sigma^{(k)})^{-1})$},
ymajorgrids,
ymin=1.61280243356402e-10, ymax=31055453.8398806,
yminorgrids,
ymode=log,
ytick style={color=black},
ytick={1e-12,1e-10,1e-08,1e-06,0.0001,0.01,1,100,10000,1000000,100000000,10000000000},
yticklabels={
  $10^{-12}$,
  $10^{-10}$,
  $10^{-8}$,
  $10^{-6}$,
  $10^{-4}$,
  $10^{-2}$,
  $10^{0}$,
  $10^{2}$,
  $10^{4}$,
  $10^{6}$,
  $10^{8}$,
  $10^{10}$
}
]
\path [fill=steelblue31119180, fill opacity=0.22]
(axis cs:1,50017.0570902689)
--(axis cs:1,0.0124824056107091)
--(axis cs:2,1.37158434594605e-08)
--(axis cs:3,1.09252094153225e-08)
--(axis cs:4,7.66800012989732e-10)
--(axis cs:5,9.20159486572354e-10)
--(axis cs:6,1.0735193994777e-09)
--(axis cs:7,1.22687931216407e-09)
--(axis cs:8,1.38023922294368e-09)
--(axis cs:9,1.53359913285762e-09)
--(axis cs:10,1.68695904247541e-09)
--(axis cs:11,1.8403189499961e-09)
--(axis cs:12,1.99367885833949e-09)
--(axis cs:13,2.1470387635824e-09)
--(axis cs:14,2.30039867079237e-09)
--(axis cs:15,2.45375857786151e-09)
--(axis cs:16,2.60711848470715e-09)
--(axis cs:17,2.76047839070479e-09)
--(axis cs:18,2.91383829642466e-09)
--(axis cs:19,3.06719820325714e-09)
--(axis cs:20,3.22055810349951e-09)
--(axis cs:21,3.3739180058555e-09)
--(axis cs:22,3.52727790984154e-09)
--(axis cs:23,3.68063780846448e-09)
--(axis cs:24,3.83399771150658e-09)
--(axis cs:25,3.9873576073462e-09)
--(axis cs:26,4.14071750682373e-09)
--(axis cs:27,4.29407740922662e-09)
--(axis cs:28,4.44743731112837e-09)
--(axis cs:29,4.60079721514906e-09)
--(axis cs:30,4.75415711534765e-09)
--(axis cs:31,4.90751701489359e-09)
--(axis cs:32,5.06087691223567e-09)
--(axis cs:33,5.21423681148531e-09)
--(axis cs:34,5.36759670706317e-09)
--(axis cs:35,5.52095659939094e-09)
--(axis cs:36,5.67431649586812e-09)
--(axis cs:37,5.82767638816077e-09)
--(axis cs:38,5.98103627736593e-09)
--(axis cs:39,6.13439615714783e-09)
--(axis cs:40,6.28775605584354e-09)
--(axis cs:41,6.44111595109066e-09)
--(axis cs:42,6.59447584426896e-09)
--(axis cs:43,6.74783573512295e-09)
--(axis cs:44,6.90119561635581e-09)
--(axis cs:45,7.05455551322534e-09)
--(axis cs:46,7.20791539852273e-09)
--(axis cs:47,7.36127528358307e-09)
--(axis cs:48,7.51463516282038e-09)
--(axis cs:49,7.66799505182395e-09)
--(axis cs:49,0.00122072676173654)
--(axis cs:49,0.00122072676173654)
--(axis cs:48,0.0012018180500137)
--(axis cs:47,0.00118330263104699)
--(axis cs:46,0.00116016233306891)
--(axis cs:45,0.00114729364029158)
--(axis cs:44,0.00116718353625313)
--(axis cs:43,0.0012032595768979)
--(axis cs:42,0.00123298095532783)
--(axis cs:41,0.00120871052898642)
--(axis cs:40,0.00119812411909815)
--(axis cs:39,0.00121772469349366)
--(axis cs:38,0.00121752107230906)
--(axis cs:37,0.00120467668933255)
--(axis cs:36,0.00117904090581859)
--(axis cs:35,0.00115030338779363)
--(axis cs:34,0.00112538414749276)
--(axis cs:33,0.00112061512201682)
--(axis cs:32,0.00112908682949272)
--(axis cs:31,0.00111098828300046)
--(axis cs:30,0.00111046317047803)
--(axis cs:29,0.00107697749553587)
--(axis cs:28,0.00107268668274715)
--(axis cs:27,0.00114495928618958)
--(axis cs:26,0.0013344007821542)
--(axis cs:25,0.00129654568141993)
--(axis cs:24,0.00127613442145623)
--(axis cs:23,0.00126962831140694)
--(axis cs:22,0.00129438602893737)
--(axis cs:21,0.00131975819268243)
--(axis cs:20,0.00137583042076577)
--(axis cs:19,0.00135288589011221)
--(axis cs:18,0.00152908738755834)
--(axis cs:17,0.00192018640221168)
--(axis cs:16,0.00203916494455734)
--(axis cs:15,0.00230261319109766)
--(axis cs:14,0.00239437088417222)
--(axis cs:13,0.00248668768321892)
--(axis cs:12,0.00234290717926231)
--(axis cs:11,0.00329736190478854)
--(axis cs:10,0.00452240652671546)
--(axis cs:9,0.00699589485594744)
--(axis cs:8,0.00945148801209277)
--(axis cs:7,0.0144098957846756)
--(axis cs:6,0.0231248464864773)
--(axis cs:5,0.0216428972375674)
--(axis cs:4,0.0340921301919517)
--(axis cs:3,0.102744437779947)
--(axis cs:2,625.16928209653)
--(axis cs:1,50017.0570902689)
--cycle;

\path [fill=darkorange25512714, fill opacity=0.22]
(axis cs:1,5250176.04145306)
--(axis cs:1,1539.3105406321)
--(axis cs:2,3.08295090675373)
--(axis cs:3,0.00160481433029423)
--(axis cs:4,8.26597624587392e-07)
--(axis cs:5,9.9191645241571e-07)
--(axis cs:6,1.15723367805047e-06)
--(axis cs:7,1.32255029782433e-06)
--(axis cs:8,1.48786840717726e-06)
--(axis cs:9,1.65318499337803e-06)
--(axis cs:10,1.81850125980978e-06)
--(axis cs:11,1.98381645089942e-06)
--(axis cs:12,2.14913281614272e-06)
--(axis cs:13,2.3144473281753e-06)
--(axis cs:14,2.47976128922924e-06)
--(axis cs:15,2.6450749098012e-06)
--(axis cs:16,2.81039091589072e-06)
--(axis cs:17,2.97570183997714e-06)
--(axis cs:18,3.14101120097011e-06)
--(axis cs:19,3.30632283744947e-06)
--(axis cs:20,3.47162994631907e-06)
--(axis cs:21,3.63694306950696e-06)
--(axis cs:22,3.80225409124798e-06)
--(axis cs:23,3.96756215829016e-06)
--(axis cs:24,4.13287312376307e-06)
--(axis cs:25,4.29817564350527e-06)
--(axis cs:26,4.4634855952835e-06)
--(axis cs:27,4.62879549017815e-06)
--(axis cs:28,4.79410514405078e-06)
--(axis cs:29,4.95941349100158e-06)
--(axis cs:30,5.12472198525038e-06)
--(axis cs:31,5.29002970910781e-06)
--(axis cs:32,5.45533347281417e-06)
--(axis cs:33,5.62064023912192e-06)
--(axis cs:34,5.78594210145117e-06)
--(axis cs:35,5.95124977344484e-06)
--(axis cs:36,6.11655303932647e-06)
--(axis cs:37,6.2818561928812e-06)
--(axis cs:38,6.44716249102688e-06)
--(axis cs:39,6.61245923981209e-06)
--(axis cs:40,6.77776759459534e-06)
--(axis cs:41,6.94307300829538e-06)
--(axis cs:42,7.10838096706705e-06)
--(axis cs:43,7.27367749562921e-06)
--(axis cs:44,7.43897757712994e-06)
--(axis cs:45,7.6042806626447e-06)
--(axis cs:46,7.76958083945955e-06)
--(axis cs:47,7.93488535149969e-06)
--(axis cs:48,8.10019082396594e-06)
--(axis cs:49,8.26549292675237e-06)
--(axis cs:49,0.2048908534813)
--(axis cs:49,0.2048908534813)
--(axis cs:48,0.205847711737532)
--(axis cs:47,0.204963591398672)
--(axis cs:46,0.201633323721067)
--(axis cs:45,0.199405383993085)
--(axis cs:44,0.195940835235183)
--(axis cs:43,0.194699868271962)
--(axis cs:42,0.196454154352906)
--(axis cs:41,0.196932433134759)
--(axis cs:40,0.193845145594517)
--(axis cs:39,0.190882433739528)
--(axis cs:38,0.188101660198002)
--(axis cs:37,0.18486422653048)
--(axis cs:36,0.181114753980323)
--(axis cs:35,0.181020895837926)
--(axis cs:34,0.182869421154537)
--(axis cs:33,0.178508406917814)
--(axis cs:32,0.177231821557031)
--(axis cs:31,0.173430139260399)
--(axis cs:30,0.173331089355177)
--(axis cs:29,0.170931266892235)
--(axis cs:28,0.168367287844292)
--(axis cs:27,0.166925159524331)
--(axis cs:26,0.163935406594053)
--(axis cs:25,0.161652622324157)
--(axis cs:24,0.157155749812123)
--(axis cs:23,0.152647467604724)
--(axis cs:22,0.150379770299458)
--(axis cs:21,0.150205506266773)
--(axis cs:20,0.147591691895458)
--(axis cs:19,0.144741620976551)
--(axis cs:18,0.139618997213487)
--(axis cs:17,0.14110297068345)
--(axis cs:16,0.13634696617193)
--(axis cs:15,0.132331839459392)
--(axis cs:14,0.127514774003657)
--(axis cs:13,0.122238466503615)
--(axis cs:12,0.118228366572648)
--(axis cs:11,0.113439820802865)
--(axis cs:10,0.124013954228983)
--(axis cs:9,0.12372982987402)
--(axis cs:8,0.126110211394228)
--(axis cs:7,0.125093643892505)
--(axis cs:6,0.163244439803291)
--(axis cs:5,0.486408103059589)
--(axis cs:4,6.31631475921521)
--(axis cs:3,850.94886810022)
--(axis cs:2,885136.262329497)
--(axis cs:1,5250176.04145306)
--cycle;

\path [fill=forestgreen4416044, fill opacity=0.22]
(axis cs:1,16265411.6092466)
--(axis cs:1,355.718308656739)
--(axis cs:2,3.38524048437722)
--(axis cs:3,0.10286047633159)
--(axis cs:4,0.0512709700526406)
--(axis cs:5,0.0474000518173553)
--(axis cs:6,0.0415207856373012)
--(axis cs:7,0.0345015660409152)
--(axis cs:8,0.0286107915405438)
--(axis cs:9,0.028805793622096)
--(axis cs:10,0.0286128934503853)
--(axis cs:11,0.0273781842205386)
--(axis cs:12,0.0252025408241166)
--(axis cs:13,0.0249685870284794)
--(axis cs:14,0.025520000083866)
--(axis cs:15,0.025790420547472)
--(axis cs:16,0.0247057390596526)
--(axis cs:17,0.0247101519411107)
--(axis cs:18,0.024378515388963)
--(axis cs:19,0.0248018664493342)
--(axis cs:20,0.0247825728316765)
--(axis cs:21,0.0244738713736087)
--(axis cs:22,0.0246662669303822)
--(axis cs:23,0.0250235138345025)
--(axis cs:24,0.0252890048228706)
--(axis cs:25,0.0256274075115006)
--(axis cs:26,0.0257561555684898)
--(axis cs:27,0.0256103662835175)
--(axis cs:28,0.0255039583034009)
--(axis cs:29,0.0254810039310617)
--(axis cs:30,0.0247638942493463)
--(axis cs:31,0.0248466920646156)
--(axis cs:32,0.025040556752017)
--(axis cs:33,0.0248013455382493)
--(axis cs:34,0.0246352307823662)
--(axis cs:35,0.0244531968608384)
--(axis cs:36,0.0245097473265062)
--(axis cs:37,0.0247678643710211)
--(axis cs:38,0.0248880220048134)
--(axis cs:39,0.0249944936897383)
--(axis cs:40,0.024962114648674)
--(axis cs:41,0.0251140988358305)
--(axis cs:42,0.0252609583389643)
--(axis cs:43,0.0252888272647413)
--(axis cs:44,0.025539732114616)
--(axis cs:45,0.0257417010464625)
--(axis cs:46,0.0257679929965904)
--(axis cs:47,0.025890559956468)
--(axis cs:48,0.0259730173863784)
--(axis cs:49,0.0258597330328244)
--(axis cs:49,0.0620743342001591)
--(axis cs:49,0.0620743342001591)
--(axis cs:48,0.0625110299422145)
--(axis cs:47,0.0636861686824507)
--(axis cs:46,0.0646208212867419)
--(axis cs:45,0.0659211355995382)
--(axis cs:44,0.0677550140443066)
--(axis cs:43,0.0684045099875968)
--(axis cs:42,0.0681060860514997)
--(axis cs:41,0.0680575850274266)
--(axis cs:40,0.0672784300171852)
--(axis cs:39,0.0670172627640752)
--(axis cs:38,0.0680271280796703)
--(axis cs:37,0.0681723459763609)
--(axis cs:36,0.0681377944623055)
--(axis cs:35,0.0687113699500238)
--(axis cs:34,0.0695387308430691)
--(axis cs:33,0.0713257075640329)
--(axis cs:32,0.0720733480144005)
--(axis cs:31,0.073140780487712)
--(axis cs:30,0.0728395532935244)
--(axis cs:29,0.073043178068463)
--(axis cs:28,0.0732896852054937)
--(axis cs:27,0.0757529748831215)
--(axis cs:26,0.0786333103155572)
--(axis cs:25,0.0822346995837322)
--(axis cs:24,0.0859568084303886)
--(axis cs:23,0.0915961852609153)
--(axis cs:22,0.0912624883030494)
--(axis cs:21,0.092879609939598)
--(axis cs:20,0.105567365986803)
--(axis cs:19,0.109165121531604)
--(axis cs:18,0.111434098989983)
--(axis cs:17,0.125368162178722)
--(axis cs:16,0.136145439970022)
--(axis cs:15,0.140391515681559)
--(axis cs:14,0.143791867991084)
--(axis cs:13,0.149060767882534)
--(axis cs:12,0.152724132569926)
--(axis cs:11,0.186033388774039)
--(axis cs:10,0.250994094949606)
--(axis cs:9,0.474679827118701)
--(axis cs:8,0.753060094039324)
--(axis cs:7,6.50312935324847)
--(axis cs:6,752.12590151464)
--(axis cs:5,3800.05056818083)
--(axis cs:4,33953.8170543493)
--(axis cs:3,798571.107104835)
--(axis cs:2,1413996.39328306)
--(axis cs:1,16265411.6092466)
--cycle;

\path [fill=crimson2143940, fill opacity=0.22]
(axis cs:1,16158506.4636865)
--(axis cs:1,97.8171547746747)
--(axis cs:2,0.0940487744658667)
--(axis cs:3,0.0164706912738156)
--(axis cs:4,0.0133960874336142)
--(axis cs:5,0.00922410780143532)
--(axis cs:6,0.00709793602139739)
--(axis cs:7,0.00670302272826322)
--(axis cs:8,0.00648186198494931)
--(axis cs:9,0.0064079769738157)
--(axis cs:10,0.00645715736719929)
--(axis cs:11,0.00658916795742855)
--(axis cs:12,0.00669553562647488)
--(axis cs:13,0.00688218367075467)
--(axis cs:14,0.00700665918525969)
--(axis cs:15,0.00716862706241525)
--(axis cs:16,0.00719840735227867)
--(axis cs:17,0.00721640114833516)
--(axis cs:18,0.00724807422594106)
--(axis cs:19,0.00730168576930837)
--(axis cs:20,0.00737943657156818)
--(axis cs:21,0.00749481327693678)
--(axis cs:22,0.00755422988378294)
--(axis cs:23,0.00755150019263013)
--(axis cs:24,0.00760598634222555)
--(axis cs:25,0.00760779310845078)
--(axis cs:26,0.00763103500705818)
--(axis cs:27,0.00769905717029502)
--(axis cs:28,0.00770736769504113)
--(axis cs:29,0.00772250871079441)
--(axis cs:30,0.00777286126175572)
--(axis cs:31,0.00776427485063402)
--(axis cs:32,0.00773988173693322)
--(axis cs:33,0.00771771628422536)
--(axis cs:34,0.00775652488649569)
--(axis cs:35,0.00776922811053311)
--(axis cs:36,0.00776305171056735)
--(axis cs:37,0.00775844168772086)
--(axis cs:38,0.00777843521033406)
--(axis cs:39,0.00777902925480202)
--(axis cs:40,0.00778190667187736)
--(axis cs:41,0.00776482350862794)
--(axis cs:42,0.00778705307556905)
--(axis cs:43,0.00780910131715425)
--(axis cs:44,0.0078089959731875)
--(axis cs:45,0.00780456977290596)
--(axis cs:46,0.00782493849948956)
--(axis cs:47,0.00784002106358221)
--(axis cs:48,0.0078549770629937)
--(axis cs:49,0.00786450790163581)
--(axis cs:49,0.0134169192065161)
--(axis cs:49,0.0134169192065161)
--(axis cs:48,0.0135340584843531)
--(axis cs:47,0.0135720289803862)
--(axis cs:46,0.0136394574182699)
--(axis cs:45,0.0136887726849127)
--(axis cs:44,0.0138122657025383)
--(axis cs:43,0.0138675485836706)
--(axis cs:42,0.0139918271636059)
--(axis cs:41,0.0140453951630227)
--(axis cs:40,0.0141512506755435)
--(axis cs:39,0.0142291307921443)
--(axis cs:38,0.0142858227704734)
--(axis cs:37,0.0143139236149889)
--(axis cs:36,0.014639909619057)
--(axis cs:35,0.0147751106185576)
--(axis cs:34,0.0148725665336551)
--(axis cs:33,0.0151417728477299)
--(axis cs:32,0.0153924643930601)
--(axis cs:31,0.0154382135179885)
--(axis cs:30,0.0159577109300667)
--(axis cs:29,0.0162947279986688)
--(axis cs:28,0.0165097829234138)
--(axis cs:27,0.0166819345055086)
--(axis cs:26,0.0167912989501017)
--(axis cs:25,0.0169371729334366)
--(axis cs:24,0.0170312759361976)
--(axis cs:23,0.0170935884095524)
--(axis cs:22,0.0172803235099589)
--(axis cs:21,0.0176555693798441)
--(axis cs:20,0.0179020898838248)
--(axis cs:19,0.0178944896493512)
--(axis cs:18,0.0182790642667772)
--(axis cs:17,0.0189714396517311)
--(axis cs:16,0.0193397439805964)
--(axis cs:15,0.0193524378807324)
--(axis cs:14,0.0191950132430813)
--(axis cs:13,0.0195858694761894)
--(axis cs:12,0.0205643751582679)
--(axis cs:11,0.022290496471129)
--(axis cs:10,0.0229395067300609)
--(axis cs:9,0.0253168332619417)
--(axis cs:8,0.0276256909650156)
--(axis cs:7,0.0313632007291823)
--(axis cs:6,0.04111991567739)
--(axis cs:5,0.0710843293407611)
--(axis cs:4,3.8086393629018)
--(axis cs:3,1327.23470690327)
--(axis cs:2,90793.184069323)
--(axis cs:1,16158506.4636865)
--cycle;

\addplot [thick, steelblue31119180, mark=*, mark size=2, mark options={solid}]
table {%
1 0.489498417923144
2 0.0102394554840501
3 0.00174027594138102
4 0.000162645029798216
5 5.86643191277299e-05
6 6.84335436001536e-05
7 6.85165125311027e-05
8 2.86386291753331e-05
9 3.18197455047914e-05
10 3.50003527583169e-05
11 3.81811304124129e-05
12 4.13616857055011e-05
13 3.53046214462984e-05
14 3.78013671493614e-05
15 3.20308706003202e-05
16 3.40324234304755e-05
17 3.60336749728878e-05
18 3.80349460716103e-05
19 4.00363769749038e-05
20 4.20377639101027e-05
21 4.4038728655711e-05
22 4.60392085532453e-05
23 4.80396431774916e-05
24 5.00407480440461e-05
25 5.20417169210119e-05
26 5.40418619084198e-05
27 5.6042175045079e-05
28 5.80420804297485e-05
29 6.00427939704534e-05
30 6.20432357421482e-05
31 6.40439354708594e-05
32 6.60442807537085e-05
33 6.80448837466712e-05
34 7.00449232061609e-05
35 7.2045257017453e-05
36 7.404541556386e-05
37 7.60460215183152e-05
38 7.80454657680375e-05
39 8.00448950980314e-05
40 8.20449590736647e-05
41 8.40447228678793e-05
42 8.60444470295937e-05
43 8.80437734189967e-05
44 9.00417815737282e-05
45 9.20413040597376e-05
46 9.40403487095779e-05
47 9.60388930843702e-05
48 9.80368516082242e-05
49 0.00010003581971905
};
\addlegendentry{Frank-Wolfe}
\addplot [thick, darkorange25512714, mark=*, mark size=2, mark options={solid}]
table {%
1 91139.3182027116
2 504.342848190118
3 0.940063964755766
4 0.0810756512743676
5 0.0572711429349562
6 0.0503103466890939
7 0.0516780719995258
8 0.0511284013268333
9 0.04966184873803
10 0.0522611796538809
11 0.0535057303043982
12 0.0548782945246797
13 0.0553139753646919
14 0.0563429351570803
15 0.0569610430865981
16 0.0584642445847924
17 0.060805958975634
18 0.062553354217457
19 0.0641671683306322
20 0.0638882409398346
21 0.0643772270568577
22 0.065628229439193
23 0.0664699100442295
24 0.0671800871802012
25 0.0677509793732046
26 0.0691616139681925
27 0.0703859111008283
28 0.0714731360431399
29 0.0725657441728706
30 0.0736741763211714
31 0.0747731791180934
32 0.0759741970031031
33 0.0767902569190758
34 0.0772809171775338
35 0.0784690273628909
36 0.0797874714618269
37 0.0803143788674175
38 0.0812506327646622
39 0.0817796510048101
40 0.082961662691046
41 0.0835545614363063
42 0.0841632307941393
43 0.0848229118725173
44 0.0848226006877772
45 0.085276090086661
46 0.08612863743687
47 0.0857874513106715
48 0.0864908104455663
49 0.087456949642861
};
\addlegendentry{Certainty equivalence}
\addplot [thick, forestgreen4416044, mark=*, mark size=2, mark options={solid}]
table {%
1 248027.763500307
2 15710.0000612975
3 249.352431885332
4 12.5473269192017
5 0.483709260033117
6 0.149165661111197
7 0.094065548833188
8 0.0712036369094902
9 0.0650979747970566
10 0.0621720300987082
11 0.057055083294357
12 0.0533751578206665
13 0.0523662739977445
14 0.0514899085256561
15 0.0512836740715739
16 0.049833822830515
17 0.0483114284577059
18 0.0466427801556186
19 0.0462348190327079
20 0.044856478098267
21 0.0444434346146472
22 0.0435690790323481
23 0.0433509687520781
24 0.0428670650043319
25 0.0424948275061552
26 0.0418959531012027
27 0.0411375005740401
28 0.0415185699461903
29 0.0414961306225763
30 0.0418162862678646
31 0.0408377280450315
32 0.0405173808459551
33 0.0405927499807791
34 0.040768713087261
35 0.0408359952704043
36 0.0408263304373765
37 0.0406754693422807
38 0.0404057098374615
39 0.0403411008638552
40 0.0401797726737233
41 0.0400053891490649
42 0.0395501959284709
43 0.0398193183036665
44 0.0393595538455094
45 0.0392190272921455
46 0.0388858203076156
47 0.0390375444211593
48 0.0386784819912684
49 0.038618397408524
};
\addlegendentry{Naive exploration}
\addplot [thick, crimson2143940, mark=*, mark size=2, mark options={solid}]
table {%
1 31942.0951993137
2 174.033252401103
3 0.490546251718058
4 0.0297073379322818
5 0.0208587443557355
6 0.016175531709491
7 0.0148920636706311
8 0.0140020232486429
9 0.0130689527351444
10 0.0122579985247109
11 0.0117918886106232
12 0.0114778079493524
13 0.0112463597005767
14 0.0110083732198743
15 0.0109227193846794
16 0.0107627963350905
17 0.0106352230307643
18 0.0105833197831494
19 0.0105123745827387
20 0.0105238970246667
21 0.0104223359894061
22 0.0102749242687582
23 0.0101114280595703
24 0.0100250158454592
25 0.00993574762764548
26 0.00984334055256201
27 0.00977029433475941
28 0.00973671524162207
29 0.00971417143782622
30 0.0096537882200654
31 0.00965496618010839
32 0.00959524038908329
33 0.00959119796321218
34 0.0095712324286441
35 0.00954436836283207
36 0.00951756275061393
37 0.00946337793066732
38 0.00944235601375368
39 0.00940524056080017
40 0.0093233926793398
41 0.00929077489385005
42 0.00925599193829537
43 0.00923060712322672
44 0.00919349486264404
45 0.00914571964637964
46 0.00914697692156044
47 0.00913865215613425
48 0.00912506431557471
49 0.00911238810834967
};
\addlegendentry{Frequency-based}
\end{axis}

\end{tikzpicture}

%% file: refs.bib
@inproceedings{wagenmaker2020active,
	title={Active learning for identification of linear dynamical systems},
	author={Wagenmaker, Andrew and Jamieson, Kevin},
	booktitle={Proc. Conf. Learn. Theory (COLT)},
	pages={3487--3582},
	year={2020},
	organization={PMLR}
}

@article{frank1956algorithm,
	title={An algorithm for quadratic programming},
	author={Frank, Marguerite and Wolfe, Philip},
	journal={Nav. Res. Logist. Q.},
	volume={3},
	number={1-2},
	pages={95--110},
	year={1956},
	publisher={Wiley Online Library}
}

@inproceedings{jaggi2013revisiting,
	title     = {Revisiting {F}rank-{W}olfe: Projection-Free Sparse Convex Optimization},
	author    = {Jaggi, Martin},
	booktitle = {Proc. 30th Int. Conf. Mach. Learn. (ICML)},
	series    = {Proceedings of Machine Learning Research},
	volume    = {28},
	number    = {1},
	pages     = {427--435},
	year      = {2013},
	publisher = {PMLR}
}

@inproceedings{wahlberg2010optimal,
	title={On optimal input design in system identification for control},
	author={Wahlberg, Bo and Hjalmarsson, H{\aa}kan and Annergren, Mariette},
	booktitle={Proc. 49th IEEE Conf. Decis. Control (CDC)},
	pages={5548--5553},
	year={2010},
	organization={IEEE}
}

@inproceedings{simchowitz2020naive,
	title={Naive exploration is optimal for online lqr},
	author={Simchowitz, Max and Foster, Dylan},
	booktitle={Proc. Int. Conf. Mach. Learn. (ICML)},
	pages={8937--8948},
	year={2020},
	organization={PMLR}
}

@article{wittenmark1995adaptive,
	title={Adaptive dual control methods: An overview},
	author={Wittenmark, Bj{\"o}rn},
	journal={Adaptive Syst. Control Signal Process.},
	pages={67--72},
	year={1995},
	publisher={Elsevier}
}

@book{aastrom2013adaptive,
	title={Adaptive control},
	author={{\AA}str{\"o}m, Karl J and Wittenmark, Bj{\"o}rn},
	year={2013},
	publisher={Courier Corporation}
}

@article{rosdahl2022dual,
	title={Dual Control by Reinforcement Learning Using Deep Hyperstate Transition Models},
	author={Rosdahl, Christian and Cervin, Anton and Bernhardsson, Bo},
	journal={IFAC-PapersOnLine},
	volume={55},
	number={12},
	pages={395--401},
	year={2022},
	publisher={Elsevier}
}

@article{wagenmaker2024optimal,
	title={Optimal exploration for model-based rl in nonlinear systems},
	author={Wagenmaker, Andrew and Shi, Guanya and Jamieson, Kevin G},
	journal={Adv. Neural Inf. Process. Syst.},
	volume={36},
	year={2024}
}

@inproceedings{mutny2023active,
	title={Active exploration via experiment design in markov chains},
	author={Mutny, Mojmir and Janik, Tadeusz and Krause, Andreas},
	booktitle={Proc. Int. Conf. Artif. Intell. Stat. (AISTATS)},
	pages={7349--7374},
	year={2023},
	organization={PMLR}
}

@inproceedings{wagenmaker2021task,
	title={Task-optimal exploration in linear dynamical systems},
	author={Wagenmaker, Andrew J and Simchowitz, Max and Jamieson, Kevin},
	booktitle={Proc. Int. Conf. Mach. Learn. (ICML)},
	pages={10641--10652},
	year={2021},
	organization={PMLR}
}

@article{bartos2026optimistic,
	title={Optimistic Online {LQR} via Intrinsic Rewards},
	author={Bartos, Marcell and Lee, Bruce D and Treven, Lenart and Krause, Andreas and D{\"o}rfler, Florian and Zeilinger, Melanie N},
	journal={arXiv preprint arXiv:2603.28938},
	year={2026}
}
